\documentclass[12pt, reqno]{amsart}
\usepackage[table]{xcolor}
\usepackage{amsmath,amssymb, amsfonts,amscd,psfrag,graphicx,stmaryrd}
\usepackage{epsfig}
\usepackage{amsthm}
\usepackage{fullpage}
\usepackage{verbatim}
\usepackage{listings}
\usepackage[table]{xcolor} 
\usepackage{mathtools}
\DeclarePairedDelimiter{\ceil}{\lceil}{\rceil}
\DeclarePairedDelimiter\floor{\lfloor}{\rfloor}
\newcommand{\lebn}

\usepackage[misc]{ifsym}
\usepackage{caption}
\usepackage{float}
\usepackage{amsmath}
\usepackage{subfigure}

\theoremstyle{plain}
\newtheorem{prop}[equation]{Proposition}

\newtheorem{thm}[equation]{Theorem}
\newtheorem{conj}[equation]{Conjecture}

\newtheorem{cor}[equation]{Corollary}
\newtheorem{lem}[equation]{Lemma}

\theoremstyle{definition}

\numberwithin{equation}{section}

\newcommand{\D}{\Delta}

\usepackage[a4paper,left=1.5cm,right=1.5cm,top=2cm,bottom=2cm]{geometry}

\allowdisplaybreaks

\usepackage{tikz}
\usepackage{verbatim}
\usetikzlibrary{shapes,shadows,calc}
\usepgflibrary{arrows}
\usetikzlibrary{arrows, decorations.markings, calc, fadings, decorations.pathreplacing, patterns, decorations.pathmorphing, positioning}
\tikzset{nodc/.style={circle,draw=blue!50,fill=pink!80,inner sep=1.6pt}}
\tikzset{nodr/.style={circle,draw=black,fill=green!50!black,inner sep=2pt}}
\tikzset{nodel/.style={circle,draw=black,inner sep=2.2pt}}
\tikzset{nodinvisible/.style={circle,draw=white,inner sep=2pt}}
\tikzset{nodpale/.style={circle,draw=gray,fill=gray,inner sep=1.6pt}}

\tikzset{nod1/.style={circle,draw=black,fill=black,inner sep=1pt}}
\tikzset{nod2/.style={circle,draw=black,fill=blue!75!black,inner sep=1.6pt}}
\tikzset{nod3/.style={circle,draw=black,fill=black,inner sep=1.8pt}}
\tikzset{noddiam/.style={diamond,draw=black,inner sep=2pt}}
\tikzset{nodw/.style={circle,draw=black,inner sep=1.8pt}}

\usepackage[colorlinks, urlcolor=blue, linkcolor=blue, citecolor=blue]{hyperref}

\makeatletter
 \def\@textbottom{\vskip \z@ \@plus 10pt}
 \let\@texttop\relax
\makeatother

\usepackage[utf8]{inputenc}
\usepackage[T1]{fontenc}

\begin{document}

\bibliographystyle{plain}

\title[2-distance coloring of planar graphs with girth five]{Some results on $2$-distance coloring of planar graphs with girth five}

\author{Zakir Deniz}

\address{Department of Mathematics, D\"uzce University, D\"uzce, 81620, Turkey.}
\email{zakirdeniz@duzce.edu.tr}

\date{\today}
\thanks{}
\subjclass[2010]{}

\begin{abstract}
A vertex coloring of a graph $G$ is called a 2-distance coloring if any two vertices at distance at most $2$ from each other receive different colors. Suppose that $G$ is a planar graph with girth $5$ and maximum degree $\D$. We prove that $G$ admits a $2$-distance $\D+7$ coloring, which improves the result of Dong and Lin (J. Comb. Optim. 32(2), 645-655, 2016).  Moreover, we prove that $G$ admits a $2$-distance $\D+6$ coloring when $\D\geq 10$.
\end{abstract}
\maketitle

\section{Introduction}

We consider only finite simple graphs throughout this paper, and refer to \cite{west} for terminology and notation not defined here. Let $G$ be a graph, we use $V(G),E(G),F(G),\D(G)$ and $g(G)$ to denote the vertex, edge and face set, the maximum degree and girth of $G$, respectively. If there is no confusion in the context, we abbreviate $\D(G),g(G)$ to $\D,g$. 
A 2-distance coloring is a proper vertex coloring where two vertices that are adjacent or have a common neighbour receive different colors, and the smallest number of colors for which $G$ admits a 2-distance coloring is known as the 2-distance chromatic number $\chi_2(G)$ of $G$. 

In 1977, Wegner \cite{wegner} posed the following conjecture.

\begin{conj}\label{conj:main}
For every planar graph $G$, $\chi_2(G) \leq 7$ if $\Delta=3$, $\chi_2(G) \leq \Delta+5$ if $4\leq \Delta\leq 7$, and $\chi_2(G) \leq  \floor[\big]{\frac{3\Delta}{2}}+1$ if $\Delta\geq 8$.
\end{conj}

Although the conjecture is still widely open, there are some partial solution: Thomassen \cite{thomassen} (independently by Hartke et al. \cite{hartke}) proved the conjecture for planar graphs with $\Delta = 3$. There are also some upper bounds for $2$-distance chromatic number of planar graphs. For instance, van den Heuvel and McGuinness \cite{van-den} showed that $\chi_2(G) \leq 2\Delta + 25$, while the bound $\chi_2(G) \leq \ceil[\big]{ \frac{5\D}{3}}+78$ was proved by Molloy and Salavatipour \cite{molloy}.\medskip

For  planar graphs  with girth restriction, La and Montassier presented a nice summary of the latest known results in \cite{la-mont-2022}. For instance, if $G$ is a planar graph with $g \geq  6$,  Bu and Zhu \cite{bu-zu} showed that $\chi_2(G) \leq \Delta + 5$,  which confirms the Conjecture \ref{conj:main} for the planar graphs with girth six. 
On the other hand, La \cite{la-2021} prove that $\chi_2(G) \leq \Delta + 3$ if either $g\geq 7$ and $\D\geq 6$ or $g\geq 8$ and $\D\geq 4$. \medskip

When $G$ is a planar graph with girth $5$, Dong and Lin \cite{dong-lin-2016} prove the following.

\begin{thm}\label{thm:8}\cite{dong-lin-2016}
If $G$ is a planar graph with  $g\geq 5$, then $\chi_2(G) \leq \Delta + 8$. 
\end{thm}

Later on, the same authors improve their result in \cite{dong-lin-2017} when $\Delta\notin \{7,8\}$.

\begin{thm}\label{thm:notin-7-8}\cite{dong-lin-2017}
If $G$ is a planar graph with  $g\geq 5$ and $\Delta\notin \{7,8\}$, then $\chi_2(G) \leq \Delta + 7$. 
\end{thm}

One of our main results is the following, which improves upon the result in Theorem \ref{thm:8} and extends the result presented in Theorem \ref{thm:notin-7-8}.

\begin{thm}\label{thm:main}
If $G$ is a planar graph with  $g\geq 5$, then $\chi_2(G) \leq \Delta + 7$. 
\end{thm}

We also improve the bound in Theorem \ref{thm:notin-7-8} by reducing one when $\D\geq 10$, which  further enhances a result presented in \cite{bu-shang}.

\begin{thm}\label{thm:main2}
If $G$ is a planar graph with  $g\geq 5$ and $\D\geq 10$, then $\chi_2(G) \leq \Delta + 6$. 
\end{thm}

\section{Structural Properties of Critical Graphs}

Given a planar  graph $G$, we denote by $\ell(f)$ the length of a face $f$ and by $d(v)$ the degree of a vertex $v$. 
A $k$-vertex is a vertex of degree $k$. A $k^{-}$-vertex is a vertex of degree at most $k$, and a $k^{+}$-vertex is a vertex of degree at least $k$. A $k$ ($k^-$ or $k^+$)-face and $k$ ($k^-$ or $k^+$)-neighbour is defined similarly. A $k(d)$-vertex is a $k$-vertex adjacent to $d$ $2$-vertices. 
When $r\leq d(v)\leq s$, the vertex $v$ is said to be $(r\text{-}s)$-vertex.

For a vertex $v\in V(G)$, we use $n_i(v)$ (resp. $n_2^k(v)$) to denote the number of $i$-vertices (resp.   $2$-vertices having a $k$-neighbour) adjacent to $v$. Let $v\in V(G)$, we define $D(v)=\Sigma_{v_i\in N(v)}d(v_i)$. 
We denote by $d(u,v)$ the distance between $u$ and $v$ for any pair $u,v\in V(G)$. Also, we set $N_i(v)=\{u\in V(G) \ | \  1 \leq d(u,v) \leq i \}$ for $i\geq 1$, and clearly $N_1(v)=N(v)$. 
If we have a path $uvw$ in a graph $G$ with $d_G(v)=2$, we say that $u$ and $w$ are \emph{weak-adjacent}. A pair of weak-adjacent vertices is said to be \emph{weak neighbour} of each other. The neighbours of a $3(1)$-vertex different from a $2$-vertex are called \emph{star-adjacent}. 
Obviously, every $2$-vertex $v$ having a $k^-$-neighbour has the property that $D(v)\leq \D+k$. \medskip

We call a vertex $v$ as \emph{$k$-expendable} if $D(v)<\Delta+k+\sum_{i=3}^{k-1} n_2^i(v)$ for $k\geq 2$. Denote by $e_k(v)$ the number of $k$-expendable vertices of distance at most $2$ from $v$. If $v$ is a vertex with $D(v)<\D+k+e_k(v)$, then it is called \emph{$k$-light},  otherwise it is called \emph{$k$-heavy}. Notice that every $2$-vertex adjacent to a $t$-vertex for $t\leq k-1$ is a $k$-expendable vertex, so $\sum_{i=3}^{k-1} n_2^i(v)\leq e_k(v)$ for any vertex $v$. In particular, every $k$-expendable vertex is a $k$-light vertex.

A graph $G$ is called \emph{$t$-critical}  if $G$ does not admit any $2$-distance coloring using $t$ colors, but any proper subgraph of $G$ does. It follows that if $G$ is $t$-critical then $G-x$ has a 2-distance coloring with $t$ colors for every $x\in V(G)\cup E(G)$.

We introduce a special coloring as follows: Given a graph $G=(V,E)$, a set $T\subseteq V$ and a positive integer $\ell$, a function $f:T \rightarrow [\ell]$ is a \emph{partial $2$-distance $\ell$-coloring} of $G$ if $f(u) \neq f(v)$ for all $u,v\in T$ with $d_G(u,v)\leq 2$. The vertices in $T$ are called \emph{colored} while  the vertices in $V\setminus T$ are called \emph{uncolored}. If a partial $2$-distance  $\ell$-coloring $f$ of a graph $G=(V,E)$ has domain $V$ (i.e., $T=V$) then it corresponds to a $2$-distance $\ell$-coloring of $G$.  \medskip

Let us now present some structural results that will be the main tools in the proofs of Theorem \ref{thm:main} and Theorem \ref{thm:main2}. The proofs of the following lemmas are quite similar. We start with a  partial $2$-distance $(\Delta+k)$-coloring of a graph $G=(V,E)$ whose colored vertex set is $W\subset V$. If $V-W$ consists of $k$-light vertices, we split $V-W$ into three groups $S_1, S_2,$ and $ S_3$. Here, $S_3$ represents the set of $2$-vertices having a $(k-1)^-$-neighbour, $S_2$ is the set $k$-expendable vertices not belonging to $S_3$, and  $S_1$ is the set $k$-light vertices not belonging to $S_2\cup S_3$. We assign an available color to each vertex in $S_1,S_2$, and $S_3$ (in this ordering). Consequently, we obtain a $2$-distance $(\Delta+k)$-coloring of the graph $G$.

\begin{lem}\label{lem:non-heavy}
Let $G=(V,E)$ be a graph with a partial $2$-distance $(\Delta+k)$-coloring whose colored vertex set is $W\subset V$ and let $v\in V\setminus W$ be a vertex satisfying $D(v)<\D+k+e_k(v)$. Then there exists a partial 2-distance $(\D+k)$-coloring of $G$ whose colored vertex set is $W\cup \{v\}$.
\end{lem}
\begin{proof}

Consider a partial 2-distance $(\D+k)$-coloring $f$ of $G$ such that its colored vertex set is $W\subset V$, let $v\in V\setminus W$ be a vertex satisfying $D(v)<\D+k+e_k(v)$. We denote by $R$ the set of $k$-expendable vertices of distance at most $2$ from $v$, and let $T$ be the set of all $2$-vertices having a $(k-1)^-$-neighbour. Recall that every $2$-vertex adjacent to a $t$-vertex for $t\leq k-1$ is a $k$-expendable, so we have $|R|=e_k(v)$ and
$(T\cap N_2(v)) \subseteq R$.

We will construct a new partial 2-distance $(\D+k)$-coloring $f'$ of $G$ such that its colored vertex set is $W\cup \{v\}$. Let the vertices in $W\setminus (R \cup T)$ keep their colors with respect to $f$ and decolor the vertices in $R \cup T$ so that we will later assign them appropriate colors. 
Now $v$ has at most $\D + k-1$ forbidden colors (i.e. the colors used in $N_2(v)$) as $D(v) < \D + k + e_k(v)$, so we assign $v$ with an available color. Our next aim is to recolor every vertex in $R \cap W$. Notice that every vertex $z \in (R\setminus T) \cap W$ has at most $\D + k-1$ forbidden colors since the vertices in $T$ are decolored yet and $D(z) < \D + k + \sum_{i=3}^{k-1} n_2^i(z)$. Thereby, we can recolor all vertices in $(R\setminus T) \cap W$ with some available colors. Next we recolor all vertices in $T \cap W$ with available colors. Thus we obtain a partial 2-distance $(\D + k)$-coloring of $G$ whose colored set is $W \cup \{v\}$. 
\end{proof}

The following is an immediate consequence of Lemma \ref{lem:non-heavy} and  provides an inductive way to recolor $k$-light vertices in a partial $2$-distance $(\Delta+k)$-coloring of a graph $G$.

\begin{cor}\label{cor:light-ext}
Let $G=(V,E)$ be a graph with a partial $2$-distance $(\Delta+k)$-coloring whose colored vertex set is $W\subset V$. If $v\in V\setminus W$  is a $k$-light vertex in $G$, then there exists a partial 2-distance $(\D+k)$-coloring of $G$ whose colored vertex set is $W\cup \{v\}$.
\end{cor}

We will now prove that every neighbour of a $k$-light vertex is a $k$-heavy vertex in a $(\D+k)$-critical graph, which will be useful throughout the paper.

\begin{lem}\label{lem:light-heavy}
If $v$ is a $k$-light vertex in a $(\D+k)$-critical graph $G$, then each neighbour of $v$ is a $k$-heavy vertex.
\end{lem}
\begin{proof}
Let $v$ be a $k$-light vertex. Assume to the contrary that there exists a $k$-light vertex $u\in N(v)$ with $D(u)< \D+k+e_k(u)$. Since $G$ is $(\D+k)$-critical, the graph  $G-uv$ has a 2-distance coloring $f$ with $\D+k$ colors. We decolor the vertices $u$ and $v$, and thereby $G$ has a partial $2$-distance $(\Delta+k)$-coloring $f'$ whose colored vertex set is $V(G)\setminus \{u,v\}$.  Consider the vertex $u$. As it is a $k$-light vertex in $G$,  there exists a partial 2-distance $(\D+k)$-coloring $f''$ of $G$ whose colored vertex set is $V(G)\setminus \{v\}$ by applying Corollary \ref{cor:light-ext}. Finally, we again apply Corollary \ref{cor:light-ext} for the vertex $v$ so that we obtain a proper 2-distance coloring of $G$ with  $\D+k$ colors, a contradiction.
\end{proof}

We point out that some vertices having $2$- or $3(1)$-neighbour in a $(\D+k)$-critical graph $G$ are always $k$-heavy as follows.

\begin{lem}\label{lem:7-heavy}
Let $G$ be a $(\D+k)$-critical graph for $k\geq 4$. Every $k^-$-vertex having a $2$-neighbour is $k$-heavy. In particular, every $(k-1)$-vertex having a $3(1)$-neighbour is $k$-heavy as well.
\end{lem}
\begin{proof}
Let $v$ be a $k^-$-vertex having a $2$-neighbour $u$. The case $d(v)\leq k-1$ follows from the fact  that $D(u)\leq \D+d(v)<\D+k$ together with Lemma \ref{lem:light-heavy}. Assume to the contrary  that $v$ is a $k$-light $k$-vertex having a $2$-neighbour $u$. Since $G$ is $(\D+k)$-critical, the graph $G-uv$ has a 2-distance coloring $f$ with $\D+k$ colors.  We first decolor $u$ and $v$. Now, $u$ has at most $\D+k-1$ forbidden colors since $D(u)\leq  \D + d(v)$ and $v$ is uncolored vertex yet, so we recolor $u$ with an available color. In addition,  we can recolor $v$ with an available color by applying Corollary \ref{cor:light-ext} since it is $k$-light vertex. Thus we obtain a proper 2-distance $(\D+k)$-coloring of $G$, a contradiction.

Suppose now that $v$ is a $(k-1)$-vertex having a $3(1)$-neighbour $u$. Let $w$ be $2$-neighbour of $u$.  By contradiction, assume  that  $v$ is a $k$-light vertex. We first decolor $u$, $v$ and $w$. Similarly as above, $u$ has at most $\D+k-1$ forbidden colors, since $v$ and $w$ are uncolored vertex and $D(u)\leq  \D + d(v)+2$, so we recolor $u$ with an available color. In addition,  we can recolor $v$ with an available color by applying Corollary \ref{cor:light-ext} since it is a $k$-light vertex. Finally, we give an available color to $w$ which has at most $\D+3$ forbidden colors with $3\leq k-1$ because it is a $2$-vertex having a $3$-neighbour. Thus we obtain a proper 2-distance $(\D+k)$-coloring of $G$, which leads to a contradiction.
\end{proof}

Given a $(\D+k)$-critical graph $G$. If a vertex $v\in V(G)$ has a $k$-expendable neighbour $u$, then $v$ would be $k$-heavy by Lemma \ref{lem:light-heavy}, and so  $D(v)\geq \D+k+e_k(v)$ with $e_k(v)\geq 1$. We show that this statement holds even if $u$ is a $k$-light vertex.

\begin{lem}\label{lem:DplusS}
Let $G$ be a $(\D+k)$-critical graph and let $v\in V(G)$. Suppose that $S$ is the set of $k$-light vertices in $N_2(v)$, and $S\cap N(v)\neq \emptyset$. If $v$ is a $k$-heavy vertex, then $D(v)\geq \D+k+|S|$.
\end{lem}

\begin{proof}
Let $v\in V(G)$, and let $S$ be the set of $k$-light vertices of distance at most $2$ from $v$. 
Suppose that $v$ is a $k$-heavy vertex. Assume for a contradiction that $D(v)<\D+k+|S|$. Denote by $T$ the set of all $k$-light vertices in $G$. Clearly, $S\subseteq T$.

Since $G$ is $(\D+k)$-critical, the graph $G-uv$ has a 2-distance coloring $f$ with $\D+k$ colors for $u\in S\cap N(v)$. We first decolor all vertices in $T$. Now, the vertex $v$ has at most $\D+k-1$ forbidden colors since $D(v)< \D+k+|S|$ and all vertices in $T$ are decolored. So we recolor $v$ with an available color. In addition,  we can easily recolor all vertices in $T$ with some  available colors by applying Corollary \ref{cor:light-ext} since $T$ consists of $k$-light vertices. Thus, we obtain a proper 2-distance coloring of $G$ with  $\D+k$ colors, a contradiction.
\end{proof}

The following is an easy consequence of Lemma \ref{lem:DplusS} together with Lemma \ref{lem:light-heavy}.

\begin{cor}\label{cor:Dplus1}
Let $G$ be a $(\D+k)$-critical graph. If $v\in V(G)$ has a $k$-light neighbour, then $D(v)\geq \D+k+1$.
\end{cor}

\section{The Proof of Theorem \ref{thm:main}} \label{sec:premECE}

\subsection{The Structure of Minimum Counterexample} \label{sub:premECE}~~\medskip

Let  $G$ be a minimal counterexample to Theorem \ref{thm:main} such that $G$ does not admit any $2$-distance coloring with $\D+7$ colors, but any proper subgraph of $G$ does. By the minimality, $G-x$ has a 2-distance coloring with $\D+7$ colors for every $x\in V(G)\cup E(G)$. So $G$ is a $(\D+7)$-critical graph. Obviously,  $G$ is connected and $\delta(G)\geq 2$. On the other hand, it suffices to consider only the case $7 \leq \Delta\leq 8$ due to Theorem \ref{thm:notin-7-8}.

Since $G$ is $(\D+7)$-critical, we have $k=7$ in the definition of $k$-light, $k$-heavy, $k$-expendable. Thus, for simplicity, we abbreviate $7$-light, $7$-heavy, $7$-expendable vertices to light, heavy and expendable vertices throughout this section.

We begin with outlining some structural properties that $G$ must carry, which will be  in use in the sequel. The following can be easily obtained from Lemma \ref{lem:light-heavy} and Corollary \ref{cor:Dplus1}.

\begin{cor}\label{cor:properties}
\begin{itemize}
\item[$(a)$] $G$ has no adjacent $2$-vertices.
\item[$(b)$] Every $2$-vertex having a $k$-neighbour for $k\leq 6$ is light.
\item[$(c)$] If a $3$-vertex has a $2$-neighbour, then its other neighbours are $6^+$-vertices.
\item[$(d)$] A $4$-vertex cannot have three $2$-neighbours.
\item[$(e)$] A $5$-vertex cannot have four $2$-neighbours.
\item[$(f)$] If a vertex $v$ has a light neighbour, then $D(v)\geq \D+8$.
\end{itemize}
\end{cor}

As $k=7$, we may rewrite Lemma \ref{lem:7-heavy} as follows.

\begin{cor}\label{cor:7-heavy}
Every $7^-$-vertex having a $2$-neighbour is heavy. In particular, every $6$-vertex having a $3(1)$-neighbour is heavy as well.
\end{cor}

\begin{prop}\label{prop:3-4-5} 
\begin{itemize}
\item[$(a)$] A $3(1)$-vertex cannot be adjacent to a $6^-$-vertex having a $2$-neighbour.
\item[$(b)$] A $4(2)$-vertex cannot be adjacent to a $5^-$-vertex having a $2$-neighbour.
\item[$(c)$] A $5(3)$-vertex cannot be adjacent to a $4^-$-vertex having a $2$-neighbour.
\end{itemize}
\end{prop}

\begin{proof}
$(a)$. Let $v$ be a $3(1)$-vertex, and let $u$ be a $2$-neighbour of $v$. It follows from Corollary \ref{cor:properties}-(c) that  each  neighbour of $v$ other than $u$ is a $6^+$-vertex. Assume that $v$ has a $6$-neighbour $w$ adjacent to a $2$-vertex $z$. By Corollary \ref{cor:7-heavy}, $v$ is a heavy vertex, and so it should be $D(v)\geq \D+7+e_7(v)$. However, $e_7(v)\geq 2$ since $u$ and $z$ are expendable vertices belonging to $N_2(v)$. This gives a contradiction with the fact that $D(v)\leq  \D+d(u)+d(w) \leq \D+8$.\medskip

$(b)$-$(c)$. The proofs are similar with $(a)$.
\end{proof}

We call a path $xyz$ as a \emph{poor path} if $7\leq d(y) \leq 8$, $2\leq d(x) \leq 3$ and $2\leq d(z) \leq 3$. If a poor path $xyz$ lies on the boundary of a face $f$, we call the vertex $y$ as \emph{f-poor vertex}.

It follows from the definition of the poor path that we can bound the number of those paths in a face.

\begin{cor}\label{cor:poor-path}
Each face $f$ has at most  $ \floor[\big]{\frac{\ell(f)}{2}}$  $f$-poor vertices.  
\end{cor}

In the rest of this section, we will apply discharging to show that $G$ does not exist. We assign to each vertex $v$ a charge $\mu(v)=\frac{3d(v)}{2}-5$ and to each face $f$ a charge $\mu(f)=\ell(f)-5$. By Euler's formula, we have
$$\sum_{v\in V}\left(\frac{3d(v)}{2}-5\right)+\sum_{f\in F}(\ell(f)-5)=-10$$

We next present some rules and redistribute accordingly. Once the discharging finishes, we check the final charge $\mu^*(v)$ and $\mu^*(f)$. If $\mu^*(v)\geq 0$ and $\mu^*(f)\geq 0$, we get a contradiction that no such counterexample graph $G$ can exist.

\subsection{Discharging Rules} \label{sub:} ~\medskip

We apply the following discharging rules.

\begin{itemize}

\item[\textbf{R1:}] Every $2$-vertex receives $1$  from each of its neighbours.
\item[\textbf{R2:}] Every $3(1)$-vertex receives $\frac{3}{4}$  from each of its $6^+$-neighbours.
\item[\textbf{R3:}] Every light $3(0)$-vertex receives $\frac{1}{6}$  from each of its neighbours.
\item[\textbf{R4:}] Let $v$ be a heavy $3(0)$-vertex.
\begin{itemize}
\item[$(a)$] If $v$ is adjacent to a light $3$-vertex, then $v$ receives $\frac{1}{3}$  from each of its $5^+$-neighbours.
\item[$(b)$] If $v$ is adjacent to a heavy $3$-vertex, then $v$ receives $\frac{1}{4}$  from each of its  $4^+$-neighbour.
\item[$(c)$] If $v$ is adjacent to no $3$-vertex, then $v$ receives $\frac{1}{6}$  from each of its neighbours.
\end{itemize}
\item[\textbf{R5:}] Let $v$ be a $4$-vertex.
\begin{itemize}
\item[$(a)$] If $v$ is a $4(1)$-vertex adjacent to exactly one $3$-vertex, then $v$ receives $\frac{1}{12}$  from each $5$-neighbour and  $\frac{1}{6}$  from each of its $6^+$-neighbours.
\item[$(b)$] If $v$ is a $4(1)$-vertex adjacent to  exactly two $3$-vertices, then $v$ receives $\frac{1}{3}$  from each of its $7^+$-neighbour.
\item[$(c)$] If $v$ is a $4(2)$-vertex, then $v$ receives  $\frac{1}{2}$  from each of its $5^+$-neighbour.
\end{itemize}

\item[\textbf{R6:}] Every $5(3)$-vertex receives $\frac{1}{4}$  from each of its $5(0)$- and $6$-neighbours; $\frac{1}{2}$  from each of its $7^+$-neighbours.

\item[\textbf{R7:}] Every $6(4)$-vertex receives $\frac{1}{6}$  from each of its $7^+$-neighbours.

\item[\textbf{R8:}] If $u$ and $v$ are adjacent $6^+$-vertices other than $6(4)$-vertex, then each of  $u,v$ gives $\frac{1}{8}$  to the faces containing $uv$. In particular, if a heavy $8(7)$-vertex $u$ is adjacent to a heavy $8$-vertex $v$ such that all neighbours of $v$ are heavy, then $u$ receives $\frac{1}{4}$ from $v$.

\item[\textbf{R9:}] Every face $f$ transfers its positive charge equally to  its incident $f$-poor vertices (see Figure \ref{fig:donation}).

\end{itemize}

Remark that if a heavy $3$-vertex $v$  has a light $3$-neighbour, then the other neighbours of $v$ must be $5^+$-vertices by Corollary \ref{cor:Dplus1}. This particularly implies that $v$ cannot have both a light $3$-neighbour and a heavy $3$-neighbour. In this way, the rules R4(a) and R4(b) cannot be applied to $v$ simultaneously.

\begin{figure}[htb]
\centering   
\begin{tikzpicture}[scale=1]
\node [nod2] at (0,.5) (x) [label=left: {\scriptsize $x$}] {};
\node [nod2] at (2,2) (y) [label=above:{\scriptsize $y$}] {}
	edge  (x);		
\node [nod2] at (4,.5) (z) [label=right:{\scriptsize $z$}] {}
	edge [] (y);
\node at (2,-.3) (t)     {$f$} ;	
		
\draw[densely dotted] (x) .. controls (1,-2) and  (3,-2) .. (z);

\draw[->, thick]  (2,.5) to (2,1.5);	

\node at (-.6,.9) (asd)     {\scriptsize $2\leq d(x)\leq 3$} ;	
\node at (4.6,.9) (asd)     {\scriptsize $2\leq d(z) \leq 3$} ;	
\node at (2,2.75) (asd)     {\scriptsize $7\leq d(y) \leq 8$};	
\end{tikzpicture}  
\caption{Donation from a $5^+$-face $f$.}
\label{fig:donation}
\end{figure}

\noindent
\textbf{Checking} $\mu^*(v), \mu^*(f)\geq 0$, for $v\in V(G), f\in F(G)$.\medskip

Clearly $\mu^*(f)\geq 0$ for each $f\in F(G)$, since every face transfers its positive charge equally to its incident $f$-poor vertices by R9. \medskip

We pick a vertex $v\in V(G)$ with $d(v)=k$. \medskip

\textbf{(1).} Let $k=2$.  By Corollary \ref{cor:properties}(a), $v$ is adjacent to two $3^+$-vertices. It follows from applying R1 that $v$ receives $1$  from each of its neighbour, and so  $\mu^*(v)\geq -2+2\times 1 =0$. \medskip 

\textbf{(2).} Let $k=3$. By Corollary \ref{cor:properties}(c), $v$ has at most one $2$-neighbour. 
In fact, if $v$ has a $2$-neighbour $u$, then the other neighbours of $v$ are $6^+$-vertices by Corollary \ref{cor:properties}(c), and so  $v$ receives $\frac{3}{4}$  from each of its $6^+$-neighbour by R2. Thus, we have  $\mu^*(v)\geq -\frac{1}{2}+2\times \frac{3}{4}- 1=0$ after $v$ sends $1$ to its $2$-neighbour by R1.

Let us now assume that $v$ has no $2$-neighbour. If $v$ is a light vertex, then $v$ receives $\frac{1}{6}$ from each of its neighbours by R3. In this case, $v$ does not have a light $3$-neighbour by Lemma \ref{lem:light-heavy}, and so we have $\mu^*(v)\geq -\frac{1}{2}+3\times \frac{1}{6}=0$. We may therefore assume that $v$ is a heavy vertex. If $v$ has a light $3$-neighbour, then the other neighbours of $v$ must be $5^+$-vertices by Corollary \ref{cor:Dplus1}. So, $v$ receives $\frac{1}{3}$ from each of its $5^+$-neighbour by R4(a). It then follows that $\mu^*(v)\geq -\frac{1}{2}+2\times \frac{1}{3}-\frac{1}{6}=0$ after $v$ sends $\frac{1}{6}$ to its light $3$-neighbour by R3.  
If $v$ has a heavy $3$-neighbour $u$, then the other neighbours of $v$ must be $4^+$-vertices since $D(v)\geq \D+7$. Note that $u$ does not receive any charge from $v$. Thus  $\mu^*(v)\geq -\frac{1}{2}+2\times \frac{1}{4}=0$ after $v$ receives $\frac{1}{4}$  from each of its $4^+$-neighbours by R4(b). 
Finally, if $v$ has no $3$-neighbour, then $\mu^*(v)\geq -\frac{1}{2}+3\times \frac{1}{6}=0$ after $v$ receives $\frac{1}{6}$  from each of its neighbours by R4(c). \medskip

\textbf{(3).} Let $k=4$. The initial charge of $v$ is $\mu(v)=\frac{3d(v)}{2}-5=1$. Observe that $v$ has no $3(1)$-neighbour by Corollary \ref{cor:properties}(c). Note that a $3^+$-vertex other than $3(1)$-vertex receives at most $\frac{1}{4}$ from each of its $4$-neighbours by R3-R4. Therefore, if $v$ has no $2$-neighbour, then $\mu^*(v)\geq 1-4\times \frac{1}{4} = 0$ by applying R3 and R4(b)-(c). So we assume that $v$ has at least one $2$-neighbour. In fact, $v$ can have at most two $2$-neighbours by Corollary \ref{cor:properties}(d). Thus $1\leq n_2(v)=t \leq 2$, and it follows from Corollary \ref{cor:7-heavy} that  $v$ is heavy, so $D(v)\geq \D+7+e_7(v)\geq \D+8\geq 15$, since $v$ has an expendable $2$-neighbour.
Denote by $x_1,\ldots, x_t$ the $2$-neighbours of $v$.\medskip

Let $n_2(v) = 1$. Clearly, $v$ can have at most two $3$-neighbours since it is heavy.  If $v$ has no any $3$-neighbour, then  $\mu^*(v)\geq 0$ after $v$ transfers $1$  to $x_1$ by R1.  We remark that if $v$ has a heavy $3$-neighbour $u$, then $D(u)\geq \D+8$ by Corollary \ref{cor:7-heavy}, since $x_1\in N_2(u)$ is an expendable vertex, and this implies that $u$ cannot have any $3^-$-neighbour, i.e., the rules R4(a)-(b) cannot be applied to $u$. 
Suppose first that $v$ has exactly one $3$-neighbour $u$. Then  $v$ has either two $5$-neighbours or at least one $6^+$-neighbour as $D(v)\geq 15$, and so $v$ receives totally at least $\frac{1}{6}$ from its $4^+$-neighbours by R5(a).  In this manner, we have $\mu^*(v)\geq 1+\frac{1}{6}-1-\frac{1}{6} =0$ after $v$ sends $1$ to $x_1$ by R1 and  $\frac{1}{6}$  to $u$ by R3 and R4(c). If $v$ has two $3$-neighbours $u$ and $w$, then the last neighbour of $v$ must be $7^+$-vertex as $D(v)\geq 15$. Then $v$ receives $\frac{1}{3}$ from its $7^+$-neighbour by R5(b). It follows that $\mu^*(v)\geq 1+\frac{1}{3}-1-2\times \frac{1}{6}= 0$ after $v$ sends $1$ to $x_1$ by R1 and  $\frac{1}{6}$  to each  of $u,w$ by R3 and R4(c). \medskip  

Let $n_2(v) = 2$. Denote by $z_1,z_2$ the neighbours of $v$ other than $x_1,x_2$. Since $v$ is heavy, and $x_1,x_2$ are expendable, we have   $D(v)\geq \D+7+e_7(v)\geq 16$, and so each of $z_1,z_2$ is a $5^+$-vertex. 
Then $v$ receives totally at least $1$ from $z_1$ and $z_2$ by R5(c).  Thus  $\mu^*(v)\geq 1+1-2\times 1 =0$ after $v$ sends $1$ to each of $x_1,x_2$ by R1.  \medskip

\textbf{(4).} Let $k=5$. The initial charge of $v$ is $\mu(v)=\frac{3d(v)}{2}-5=\frac{5}{2}$. Note that a $3^+$-vertex receives at most $\frac{1}{2}$ from its $5$-neighbour by R3-R6. So, if $v$ has no $2$-neighbour, then we have $\mu^*(v)\geq \frac{5}{2}-5\times \frac{1}{2}=0$ after $v$ sends at most $\frac{1}{2}$ to each of its neighbours by R3-R6. We may therefore assume that $1\leq n_2(v)=t \leq 3$ where the last inequality comes from Corollary \ref{cor:properties}(e). In particular, $v$ is heavy by Corollary \ref{cor:7-heavy}, and  $D(v)\geq \D+8$ by Corollary \ref{cor:Dplus1}, since $v$ has at least one $2$-neighbour. Remark that $v$  cannot have any $3(1)$- and $4(2)$-neighbours by Proposition \ref{prop:3-4-5}(a)-(b). Denote by $x_1,x_2,\ldots,x_t$ the $2$-neighbours of $v$. \medskip

Let $n_2(v) = 1$. 
Then $\mu^*(v)\geq \frac{5}{2} -1-4\times \frac{1}{3}> 0$ after $v$ sends $1$ to $x_1$ by R1;  at most $\frac{1}{3}$  to each of its other neighbours by  R3, R4(a)-(c), R5(a). \medskip

Let $n_2(v) = 2$. Recall that $v$ cannot have any $4(2)$-neighbour.  Moreover, at most two neighbours of $v$  can be $3$-vertices since $v$ is heavy.  If $v$ has two $4^+$-neighbours, then  $\mu^*(v)\geq \frac{5}{2} -2\times 1- \frac{1}{3}-2\times \frac{1}{12}= 0$ after $v$ sends $1$ to each $x_i$ by R1; at most $\frac{1}{3}$  to each of its $3$-neighbour (if exists) by  R3-R4; at most $\frac{1}{12}$  to each of its $4(1)$-neighbours (if exists) by  R5(a). We may therefore suppose that $v$ has exactly two $3$-neighbours $u$ and $w$, which are clearly $3(0)$-vertices by Proposition \ref{prop:3-4-5}(a). 
In particular, the last neighbour of $v$ is a $6^+$-vertex, since $v$ is heavy. 
Clearly, at least one of $u$ and $w$ is heavy, otherwise we would have $D(v)\geq \D+7+|\{x_1,x_2,u,w\}|= \D+11$ by Lemma \ref{lem:DplusS}, a contradiction with the fact that $D(v)\leq \D+3+3+2+2=\D+10$.  This implies that at least one of $u$ and $w$ is heavy, say $u$. Note that $u$ cannot be adjacent to any $3$-vertex since $v$ has two expendable $2$-neighbours. Thus, $\mu^*(v)\geq \frac{5}{2} -2\times 1- \frac{1}{3}- \frac{1}{6}=0$ after $v$ sends $1$ to each $x_i$ by R1;  $\frac{1}{6}$  to $u$ by  R3, R4(c); at most $\frac{1}{3}$  to $w$  by R3 and R4.  \medskip

Let $n_2(v) = 3$.  The remaining neighbours of $v$ other than $x_1,x_2,x_3$ are $4^+$-vertices, as $v$ is heavy. Denote by $z_1$ and $z_2$ the $4^+$-neighbours of $v$. Clearly, $d(z_1)+d(z_2)\geq \D+4\geq 11$, since $v$ is heavy, and it has three expendable $2$-neighbours. By the same reason, if one of $z_1$ or $z_2$ is a $5$-vertex having a $2$-neighbour or a $4$-vertex, then the other must be a $7^+$-vertex. 
First suppose that each of $z_1$ and $z_2$ is a $5(0)$- or $6^+$-vertex. Then $v$ receives at least $\frac{1}{4}$  from each of $z_1$ and $z_2$ by R6, and so $\mu^*(v)\geq \frac{5}{2}+2\times \frac{1}{4} -3\times 1= 0$  after $v$ transfers $1$  to each $x_i$ by R1. Now, we suppose that one of $z_1 $ or $z_2$ is  a $5$-vertex having a $2$-neighbour or a $4$-vertex. Thus $\mu^*(v)\geq \frac{5}{2}+ \frac{1}{2} -3\times 1= 0$, after  $v$ receives $\frac{1}{2}$  from its $7^+$-neighbour by R6, and sends $1$  to each $x_i$ by R1.\medskip

\textbf{(5).} Let $k=6 $. We have $\mu(v)=\frac{3d(v)}{2}-5=4$. Notice first that if $v$ has a $2$- or $3(1)$-neighbour, then $v$ must be heavy by Corollary \ref{cor:7-heavy}. This infers that $v$ cannot have five $2$-neighbours or six $3(1)$-neighbours. Moreover, if $v$ has a $2$-neighbour, then it cannot have any $3(1)$-neighbour by Proposition \ref{prop:3-4-5}(a). 

Suppose first that $v$ has no $2$-neighbour.   If $v$ has five $3(1)$-neighbours, then the last neighbour must be $4(0)$-, $5(0)$-, $5(1)$- or $6^+$-vertex,  since $v$ is heavy by Corollary \ref{cor:7-heavy}. Thus $\mu^*(v)\geq 4-5\times \frac{3}{4}-2\times\frac{1}{8}=0$ after $v$ sends $\frac{3}{4}$ to each of its $3(1)$-neighbour by R2, and $\frac{1}{8}$ to each face containing $vu$ by R8 when $v$ has a $6^+$-neighbour $u$ other than $6(4)$-vertex.  If $v$ has  at most four $3(1)$-neighbours, then  $\mu^*(v)\geq 4-4\times \frac{3}{4}-2\times\frac{1}{2}=0$ after  $v$ sends  $\frac{3}{4}$ to each of its $3(1)$-neighbour by R2;  at most $\frac{1}{2}$ to each of its other $5^-$-neighbour by R3-R6;  $\frac{1}{8}$ to each face containing $vu$ by R8 when $v$ has a $6^+$-neighbour $u$ other than $6(4)$-vertex.

Next suppose that $v$ has a $2$-neighbour, and so $v$ cannot have any $3(1)$-neighbour as earlier stated. By Corollary \ref{cor:7-heavy}, $v$ is heavy.  If $n_2(v)=1$, then we have $\mu^*(v)\geq 4-1-5\times \frac{1}{2}>0$ after $v$ sends  $1$  to its $2$-neighbour by R1;   at most $\frac{1}{2}$ to each of its other $5^-$-neighbours by R3-R6; $\frac{1}{8}$ to each face containing $vu$ by R8 when $v$ has a $6^+$-neighbour $u$ other than $6(4)$-vertex.  
Suppose now that $2\leq n_2(v) \leq 3$. Similarly to Proposition \ref{prop:3-4-5}(b), we deduce that a $4(2)$-vertex cannot be adjacent to a $6$-vertex having two $2$-neighbour, since a $4(2)$-vertex is heavy by Corollary \ref{cor:7-heavy}. Thus, $v$ has no $4(2)$-neighbour, and so $\mu^*(v)\geq 4-3\times 1-3\times \frac{1}{3}= 0$ after $v$ sends  $1$  to each of its $2$-neighbours by R1;   at most $\frac{1}{3}$ to each of its other $5^-$-neighbours by R3-R6; $\frac{1}{8}$ to each face containing $vu$ by R8 when $v$ has a $6^+$-neighbour $u$ other than $6(4)$-vertex.

We may further suppose that $n_2(v)=4$. Recall that the rule R8 cannot be applied to $v$, since it is a $6(4)$-vertex.  Let $z_1$ and $z_2$ be $3^+$-neighbours of $v$ with $d(z_1)\geq d(z_2) $. Since $v$ is heavy, it cannot have any $5(3)$-neighbour. In particular, we have $d(z_1)+d(z_2)\geq \D+3\geq 10$, since $v$ has four expendable $2$-neighbours. This implies that $z_1$ is a $5^+$-vertex other than $5(3)$-vertex.  On the other hand, we can say that if $z_1$ is a $7^+$-vertex, then $z_2$ would be a $3(0)$-vertex having no $3$-neighbour or a $4(0)$-vertex or a $4(1)$-vertex or a $5(r)$-vertex with $0\leq r \leq 2$ or a $6^+$-vertex  by Lemma \ref{lem:DplusS}. In such a case, $v$ receives $\frac{1}{6}$ from $z_1$ by applying R7, and sends at most $\frac{1}{6}$ to $z_2$ by R3-R6.  Thus we have $\mu^*(v)\geq 4-4\times 1+\frac{1}{6}-\frac{1}{6}=0$  after $v$ sends  $1$  to each of its $2$-neighbour by R1. 
If $z_1$ is a $6^-$-vertex, then $z_2$ would be a $4(0)$-, $5(0)$-, $5(1)$- or $6$-vertex.  Therefore $\mu^*(v)\geq 4-4\times 1=0$  after $v$ sends  $1$  to each of its $2$-neighbours by R1. \medskip

\textbf{(6).} Let $k=7$. The initial charge of $v$ is $\mu(v)=\frac{3d(v)}{2}-5=\frac{11}{2}$. We start by noting that any $3^+$-vertex  receives at most $\frac{3}{4}$ from its $7$-neighbour by R2-R7. In particular, any vertex other than $2$- and $3(1)$-vertex receives at most $\frac{1}{2}$ from its $7$-neighbour by R3-R7. It then follows that if $v$ has no $2$-neighbour, then we have  $\mu^*(v)\geq \frac{11}{2}-7\times \frac{3}{4}>  0$ after  $v$ transfers at most $\frac{3}{4}$  to each of its neighbours by R2-R7. On the other hand, if $n_2(v)=1$, then we have $\mu^*(v)\geq \frac{11}{2}-1-6\times \frac{3}{4}= 0$ after  $v$ transfers $1$  to its $2$-neighbour by R1;  at most $\frac{3}{4}$ to each of its $3^+$-neighbours by R2-R7, in particular, if $v$ has a $6^+$-neighbour $u$ other than $6(4)$-vertex, then $v$ sends $\frac{1}{8}$ to each face containing $vu$ by R8 instead of sending $\frac{3}{4}$ to $u$.  So, we may assume that $2\leq n_2(v)=t \leq 7$. Denote by $x_1,x_2,\ldots,x_t$ the $2$-neighbours of $v$, and let $y_i$ be the neighbour of $x_i$ different from $v$. Let $z_1,z_2,\ldots,z_{7-t}$ be the neighbour of $v$ other than $x_1,x_2,\ldots,x_t$. Clearly, $v$ is heavy by Corollary \ref{cor:7-heavy}.  Let $f_1,f_2\ldots,f_7$ be faces incident to $v$, and suppose that they have a clockwise order on the plane.
Remark that if $v$ is weak-adjacent to a $6^-$-vertex $y_j$, then $x_j$ would be an expendable (also light) vertex. Moreover, this forces that all $x_i$'s are light vertices as well. In such a case, $v$ would have two light neighbours as $t\geq 2$, and this implies that $v$ cannot have any  $3(1)$-neighbour by Lemma \ref{lem:DplusS} and Corollary \ref{cor:7-heavy}. In other words, if $v$ has a $3(1)$-neighbour, then  each $y_i$ must be a $7^+$-vertex. By the same reason, $v$ is not star-adjacent to any $5^-$ or $6(4)$-vertex.\medskip 

Let $n_2(v)=2$.  If $v$ has a neighbour $u$ different from $2$- and $3(1)$-vertices, then we have  $\mu^*(v)\geq \frac{11}{2}-2\times 1-4\times \frac{3}{4}-\frac{1}{2}= 0$ after  $v$ transfers $1$  to  each $x_i$ by R1;  at most $\frac{3}{4}$ to each of its neighbour other than $x_i$'s and $u$ by R2-R7; at most $\frac{1}{2}$ to $u$ by R3-R7 when $u$ is a $5^-$- or a $6(4)$-vertex; $\frac{1}{8}$ to each face containing $vu$ by R8 when $u$ is a $6^+$-vertex different from $6(4)$-vertex.  Thus we further suppose that all neighbours of $v$ other than $x_1$ and $x_2$ are $3(1)$-vertices. Denote by $w_i$ the $3^+$-neighbour of each $z_i$ different from $v$. Recall that each $y_i$ is a $7^+$-vertex, and each $w_i$ is a $6^+$-vertex different from $6(4)$ as earlier stated. 
 We will now show that $v$ receives totally at least $\frac{1}{4}$ from $f_1,f_2\ldots,f_7$.
Since $v$ is weak-adjacent to two $7^+$-vertices and star-adjacent to five $6^+$-vertices other than $6(4)$-vertex,  there exists a face $f_{i}$ for $i\in [7]$ such that  it is either a $5$-face having two adjacent $6^+$-vertices $x$ and $y$ other than $6(4)$-vertices  or a $6^+$-face.
If $f_i$ is a $6^+$-face, then it transfers at least $\frac{1}{3}$ to $v$ by R9 together with Corollary \ref{cor:poor-path}. On the other hand, if  $f_i$ is a $5$-face having two adjacent $6^+$-vertices $x$ and $y$ different from $6(4)$-vertex, then each of $x,y$ gives $\frac{1}{8}$ to the faces containing $xy$ by R8, so $f_i$ gets totally at least $\frac{1}{4}$ from $x,y$ and transfers it to $v$ by R9.
Consequently, $v$ receives at least $\frac{1}{4}$ from $f_i$. Hence $\mu^*(v)\geq \frac{11}{2}+\frac{1}{4}-2\times 1-5\times \frac{3}{4}= 0$  after $v$ transfers $1$  to each $x_i$ by R1; $\frac{3}{4}$  to  each of its $3(1)$-neighbours by R2. \medskip

Let $n_2(v)=3$. First,  suppose that $v$ has a $6^+$-neighbour $u$ different from $6(4)$-vertex. So, we have $\mu^*(v)\geq \frac{11}{2}-3\times 1-3\times \frac{3}{4}-\frac{1}{4}= 0$ after  $v$ transfers $1$  to each $x_i$ by R1; at most $\frac{3}{4}$ to each of its $3^+$-neighbours other than $u$  by R2-R7; $\frac{1}{8}$ to each face containing $vu$ by R8. 
Suppose now that all the neighbours of $v$ consist of $5^-$-vertices and $6(4)$-vertices.
If $v$ has a $6(4)$-neighbour, then, similarly as above, we have $\mu^*(v)\geq \frac{11}{2}-3\times 1-3\times \frac{3}{4}-\frac{1}{6}> 0$ after  $v$ transfers $1$  to each $x_i$ by R1; at most $\frac{3}{4}$ to each of its $5^-$-neighbours other than $x_i$'s  by R2-R6; $\frac{1}{6}$ to each of its $6(4)$-neighbours by R7. 
We may therefore assume that each neighbour of $v$ is a $5^-$-vertex. If $v$ has at most two $3(1)$-neighbours, then $\mu^*(v)\geq \frac{11}{2}-3\times 1-2\times \frac{3}{4}-2\times\frac{1}{2}=0$ after  $v$ transfers $1$  to each $x_i$ by R1; at most $\frac{3}{4}$ to each of its $3(1)$-neighbours  by R2; at most $\frac{1}{2}$ to each of its other neighbours by R3-R6. Thus we further suppose that $v$ has at least three $3(1)$-neighbours. 
Denote by $w_i$ the $3^+$-neighbour of each $z_i$ different from $v$. 
Recall that each $y_i$ is a $7^+$-vertex, and each $w_i$ is a $6^+$-vertex different from $6(4)$ as earlier stated.  
Since $v$ is weak-adjacent to three $7^+$-vertices and star-adjacent to three $6^+$-vertex different from $6(4)$-vertex, we conclude that there exist  at least two faces $f_{i_1},f_{i_2} $ for $i_1,i_2\in [7]$ 
such that each of them is either a $5$-face having two adjacent $6^+$-vertices other than $6(4)$-vertices  or a $6^+$-face. 
Similarly as above, each of $f_{i_1},f_{i_2}$ transfers at least $\frac{1}{4}$ to $v$ by applying R8 and R9 together with Corollary \ref{cor:poor-path}. Hence, $\mu^*(v)\geq \frac{11}{2}+2\times\frac{1}{4}-3\times 1-4\times \frac{3}{4}= 0$ after  $v$ transfers $1$  to each $x_i$ by R1;  $\frac{3}{4}$ to each of its $3(1)$-neighbours by R2; at most $\frac{1}{2}$ to each of its other neighbours by R3-R6. \medskip

Let $n_2(v)=4$. If $v$ has no $3(1)$-neighbour, then we have  $\mu^*(v)\geq \frac{11}{2}-4\times 1-3\times \frac{1}{2}=0$  after $v$ sends  $1$  to each $x_i$ by R1;   at most $\frac{1}{2}$ to each of its  $5^-$- and $6(4)$-neighbours other than $x_i$'s by R3-R7; $\frac{1}{8}$ to each face containing $vu$ by R8 when $v$ has a $6^+$-neighbour $u$ other than $6(4)$-vertex. We therefore suppose that $v$ has a $3(1)$-neighbour, say $z_1$.  Recall that  each $y_i$ is a $7^+$-vertex, and the neighbour of $z_1$ other than $v$ and $2$-vertex is a $6^+$-vertex different from $6(4)$-vertex as earlier stated.  Namely, $v$ is weak-adjacent to four $7^+$-vertices and star-adjacent to at least one $6^+$-vertex different from $6(4)$-vertex.
If $v$ has exactly one (resp. two) $3(1)$-neighbour, then we conclude that there exist two (resp. three) faces incident to $v$  such that each of them is either a  $5$-face having two adjacent $6^+$-vertices  other than $6(4)$-vertices  or a $6^+$-face. 
Similarly as above,  each of those faces transfers at least $\frac{1}{4}$ to $v$ by applying R8 and R9 together with Corollary \ref{cor:poor-path}. Hence, $\mu^*(v)\geq \frac{11}{2}+2\times\frac{1}{4}-4\times 1-2\times \frac{3}{4}-\frac{1}{2}= 0$ or $\mu^*(v)\geq \frac{11}{2}+3\times\frac{1}{4}-4\times 1-3\times \frac{3}{4}= 0$ after  $v$ transfers $1$  to each $x_i$ by R1;  $\frac{3}{4}$ to each of its $3(1)$-neighbours  by R2;  at most $\frac{1}{2}$ to each of its $5^-$- and $6(4)$-neighbours other than $x_i$'s and $3(1)$-vertex by R3-R7; $\frac{1}{8}$ to each face containing $vu$ by R8 when $v$ has a $6^+$-neighbour $u$ other than $6(4)$-vertex. \medskip

Let $n_2(v)= 5$. We first consider the case that $v$ is weak-adjacent to a $6^-$-vertex $y_j$. Then $x_j$ would be an expendable (also light) vertex, and this forces that all $x_i$'s are light vertices as well. By Corollary \ref{cor:7-heavy}, $v$ is a heavy vertex. Besides, $v$ cannot have a neighbour $u$ forming a $3(1)$-vertex or a $3(0)$-vertex having a light $3$-neighbour or a $4(2)$-vertex or a $4(1)$-vertex having a $3$-neighbour or a $5(3)$-vertex, since otherwise, it would be  $D(u)<\D+7+|S|$, where $S$ is the set of light vertices of distance at most $2$ from $u$,  a contradiction by Lemma \ref{lem:DplusS}. 
It follows that each neighbour of $v$ other than $x_i$'s receives at most $\frac{1}{4}$ from $v$ by R3-R7. Therefore, we have  $\mu^*(v)\geq \frac{11}{2}-5\times 1-2\times \frac{1}{4}=0$   after $v$ sends  $1$  to each $x_i$ by R1;   at most $\frac{1}{4}$ to each of its $5^-$- and $6(4)$-neighbours other than $x_i$'s by R3-R7; $\frac{1}{8}$ to each face containing $vu$ by R8 when $v$ has a $6^+$-neighbour $u$ other than $6(4)$-vertex. We may further assume that each $y_i$ is a $7^+$-vertex.

Suppose that $v$ has no $3(1)$-neighbour. Then, there exist $f_{i_1},f_{i_2},f_{i_3}$ for $i_1,i_2,i_3\in [7]$ such that each of them is either a  $5$-face having two adjacent $6^+$-vertices  other than $6(4)$-vertices  or a $6^+$-face. Similarly as above,  each of those faces transfers at least $\frac{1}{4}$ to $v$ by applying R8 and R9 together with Corollary \ref{cor:poor-path}. Hence, $\mu^*(v)\geq \frac{11}{2}+3\times\frac{1}{4}-5\times 1-2\times \frac{1}{2}>0$  after $v$ sends  $1$  to each $x_i$ by R1;   at most $\frac{1}{2}$ to each of its $5^-$ and $6(4)$-neighbour other than $x_i$'s by R3-R7; $\frac{1}{8}$ to each face containing $vu$ by R8 when $v$ has a $6^+$-neighbour $u$ other than $6(4)$-vertex.

We now suppose that $v$ has a $3(1)$-neighbour, say $z_1$.  Recall that  each $y_i$ is a $7^+$-vertex, and the neighbour of $z_1$ other than $v$ and $2$-vertex is a $6^+$-vertex different from $6(4)$-vertex as earlier stated. Since $v$ is weak-adjacent to five $7^+$-vertices and star-adjacent to at least one $6^+$-vertex different from $6(4)$-vertex, we conclude that there exist four faces $f_{i_1},f_{i_2},f_{i_3},f_{i_4}$ for $i_1,i_2,i_3,i_4\in [7]$ such that each of them is either a  $5$-face having two adjacent $6^+$-vertices other than $6(4)$-vertices  or a $6^+$-face.
Each of those faces transfers at least $\frac{1}{4}$ to $v$ by applying R8 and R9 together with Corollary \ref{cor:poor-path}. Hence, $\mu^*(v)\geq \frac{11}{2}+4\times\frac{1}{4}-5\times 1-2\times \frac{3}{4}= 0$ after  $v$ transfers $1$  to each $x_i$ by R1;  at most $\frac{3}{4}$  to each of its  $5^-$ and $6(4)$-neighbours other than $x_i$'s by R2-R7; $\frac{1}{8}$ to each face containing $vu$ by R8 when $v$ has a $6^+$-neighbour $u$ other than $6(4)$-vertex. \medskip

Let $n_2(v)\geq 6$. Notice first that if $v$ is weak-adjacent to a $6^-$-vertex $y_j$, then $x_j$ would be an expendable (also light) vertex. This forces that all $x_i$'s are light vertices as well. By Corollary \ref{cor:7-heavy}, $v$ is a heavy vertex. However, we have $D(v)\leq \D+6+|S|$, where $S$ is the set of light vertices of distance at most $2$ from $v$, a contradiction by Lemma \ref{lem:DplusS}. Thus  each $y_i$ must be a $7^+$-vertex. Similarly as above, if $n_2(v)=6$ (resp. $7$), then there exist five (resp. seven) faces incident to $v$ such that each of them is either a  $5$-face having two adjacent $6^+$-vertices  other than $6(4)$-vertices  or a $6^+$-face.
Each of those faces transfers at least $\frac{1}{4}$ to $v$ by applying R8 and R9 together with Corollary \ref{cor:poor-path}. Hence, we have $\mu^*(v)\geq \frac{11}{2}-6\times 1-\frac{3}{4}+5\times\frac{1}{4}= 0$  or $\mu^*(v)\geq \frac{11}{2}-7\times 1+7\times\frac{1}{4}> 0$  after  $v$ transfers $1$  to each $x_i$ by R1;  at most $\frac{3}{4}$ to each of its $5^-$- and $6(4)$-neighbours other than $x_i$'s by R2-R7; $\frac{1}{8}$ to each face containing $vu$ by R8 when $v$ has a $6^+$-neighbour $u$ other than $6(4)$-vertex. \medskip

\textbf{(7).} Let $k=8$. The initial charge of $v$ is $\mu(v)=\frac{3d(v)}{2}-5=7$. 
We first remark that if $v$ and all neighbours of $v$ are heavy, and $v$ is adjacent to a heavy $8(7)$-vertex $u$, then $v$ gives $\frac{1}{4}$ to $u$ by R8 as well as $\frac{1}{8}$ to each face containing $vu$ by R8. Thus, $v$ loses totally $\frac{1}{2}$ due to its heavy $8(7)$-neighbour $u$ when all neighbours of $u$ are heavy. 

Observe that $\mu^*(v)\geq 7-4\times 1-4\times\frac{3}{4}= 0$  if $v$ has at most four $2$-neighbours by R1-R8. So, we may assume that $5\leq n_2(v)=t \leq 8$. Denote by $x_1,x_2,\ldots,x_t$ the $2$-neighbours of $v$, and let $y_i$ be the other neighbour of $x_i$ different from $v$. Let $z_1,z_2,\ldots,z_{8-t}$ be the neighbours of $v$ other than $x_1,x_2,\ldots,x_t$. Notice that if $v$ is weak-adjacent to two $6^-$-vertices $y_p$ and $y_q$, then both $x_p$ and $x_q$ are expendable (also light) vertices. This forces that all $x_i$'s are light as well. In such a case,  $v$ cannot have a  $3(1)$-vertex $u$, since otherwise, it would be $D(u)\leq\D+4+|S|$, where $S$ is the set of light vertices of distance at most $2$ from $u$, a contradiction by Lemma \ref{lem:DplusS} and Corollary \ref{cor:7-heavy}. This infers that if $v$ has a $3(1)$-neighbour, then at most one of $x_i$'s is expendable.
Let $f_1,f_2\ldots,f_8$ be faces incident to $v$, and suppose that they have a clockwise order on the plane.   \medskip

Let $n_2(v)=5$. If $v$ has no $3(1)$-neighbour, then $\mu^*(v)\geq 7-5\times 1-3\times\frac{1}{2}> 0$  after $v$ transfers $1$  to each $x_i$ by R1; at most $\frac{1}{2}$  to each of its neighbours other than $x_i$'s by R3-R8; $\frac{1}{8}$ to each face containing $vu$ by R8 when $v$ has a $6^+$-neighbour $u$ other than $6(4)$-vertex. 
We now suppose that $v$ has a $3(1)$-neighbour $u$. Note that the neighbour of $u$ other than $v$ and $2$-vertex is a $6^+$-vertex different from $6(4)$-vertex by Corollary \ref{cor:properties}(c) and Proposition \ref{prop:3-4-5}(a).
Since $v$ has a $3(1)$-neighbour $u$, the vertex $v$ has at most one expendable $2$-neighbour, and this means that $v$ is weak-adjacent to at least four $7^+$-neighbours.  
Since $v$ is weak-adjacent to four $7^+$-vertices and star-adjacent to a $6^+$-vertex other than $6(4)$-vertex, there exists a face $f_{i} $ for $i\in [8]$ such that it is either a  $5$-face having two adjacent $6^+$-vertices other than $6(4)$-vertices  or a $6^+$-face. By R8 and R9 together with Corollary \ref{cor:poor-path}, $f_i$ transfers at least $\frac{1}{4}$ to $v$. Thus, $\mu^*(v)\geq 7+\frac{1}{4}-5\times 1-3\times\frac{3}{4}> 0$ after $v$ transfers $1$  to each $x_i$ by R1; at most $\frac{3}{4}$   to each of its neighbours other than $x_i$'s by R2-R8; $\frac{1}{8}$ to each face containing $vu$ by R8 when $v$ has a $6^+$-neighbour $u$ other than $6(4)$-vertex. \medskip

Let $n_2(v)=6$. If $v$ has no $3(1)$-neighbour, then $\mu^*(v)\geq 7-6\times 1-2\times\frac{1}{2}= 0$   after $v$ sends $1$  to each $x_i$ by R1; at most $\frac{1}{2}$  to each of its neighbours other than $x_i$'s by R2-R7; $\frac{1}{8}$ to each face containing $vu$ by R8 when $v$ has a $6^+$-neighbour $u$ other than $6(4)$-vertex. We may further suppose that $v$ has a $3(1)$-neighbour. This implies that the vertex $v$ has at most one expendable $2$-neighbour, i.e., $v$ is weak-adjacent to at least five $7^+$-neighbours.
Since $v$ is weak-adjacent to five $7^+$-vertices,  there exist two faces $f_{i_1},f_{i_2}$ for $i_1,i_2\in [8]$ such that each of them is either a  $5$-face having two adjacent $6^+$-vertices other than $6(4)$-vertices  or a $6^+$-face.   It follows that each of those faces transfers at least $\frac{1}{4}$ to $v$ by R8 and R9 together with Corollary \ref{cor:poor-path}. Thus, $\mu^*(v)\geq 7+2\times\frac{1}{4}-6\times 1-2\times\frac{3}{4}= 0$  after $v$ sends $1$  to each $x_i$ by R1; at most $\frac{3}{4}$ to each of its neighbours other than $x_i$'s by R2-R7;  $\frac{1}{8}$ to each face containing $vu$ by R8 when $v$ has a $6^+$-neighbour $u$ other than $6(4)$-vertex. \medskip

Let $n_2(v)=7$. Note that if $v$ is weak-adjacent to two $6^-$-vertices $y_p$ and $y_q$, then both  $x_p$ and $x_q$ would be expendable (also light) vertices, and this way, all $x_i$'s are light vertices as well. By Lemma \ref{lem:light-heavy}, $v$ is heavy, and so $v$ must have a heavy $8$-neighbour, say $z$, by Lemma \ref{lem:DplusS}. In particular, all neighbours of $z$ must be heavy as well, since otherwise, it would be $D(v)\leq \D+6+|S|$, where $S$ is the set of light vertices of distance at most $2$ from $v$, a contradiction by Lemma \ref{lem:DplusS}. It then follows from applying R8 that $v$ receives $\frac{1}{4}$ from $z$. 
Thus, $\mu^*(v)\geq 7+\frac{1}{4}-7\times 1-\frac{1}{4}= 0$    after $v$ sends $1$  to each $x_i$ by R1;  $\frac{1}{8}$ to each face containing $vz$ by R8. We may therefore suppose that $v$ is weak-adjacent to at most one $6^-$-vertex. Namely, $v$ is weak-adjacent to six $7^+$-vertices, and this implies that there exist four faces incident to $v$ such that each of them is either a  $5$-face having two adjacent $6^+$-vertices other than $6(4)$-vertices  or a $6^+$-face.   It follows that $v$ receives at least  $\frac{1}{4}$ from each of those faces by R8 and R9 together with Corollary \ref{cor:poor-path}.  Thus, $\mu^*(v)\geq 7+4\times\frac{1}{4}-7\times 1-\frac{3}{4}> 0$ after $v$ sends $1$  to each $x_i$ by R1; at most $\frac{3}{4}$ to each of its neighbours other than $x_i$'s by R2-R8;  $\frac{1}{8}$ to each face containing $vu$ by R8 when $v$ has a $6^+$-neighbour $u$ other than $6(4)$-vertex. \medskip

Let $n_2(v)=8$. Recall that if $v$ is weak-adjacent to two $6^-$-vertices, then all $x_i$'s would be light vertices. However, this is not possible by Lemma \ref{lem:light-heavy} and Lemma \ref{lem:DplusS}.
Thus, $v$ is weak-adjacent to at most one $6^-$-vertex. Namely, $v$ is weak-adjacent to six $7^+$-vertices, and this implies that there exist six faces incident to $v$ such that each of them is either a  $5$-face having two adjacent $6^+$-vertices other than $6(4)$-vertices  or a $6^+$-face.   It follows that $v$ receives at least  $\frac{1}{4}$ from each of those faces by R8 and R9 together with Corollary \ref{cor:poor-path}.  Thus, $\mu^*(v)\geq 7+6\times\frac{1}{4}-8\times 1> 0$  after $v$ sends $1$ to its each neighbours by R1. \newpage





\section{The Proof of Theorem \ref{thm:main2}} \label{sec:premECE2}

\subsection{The Structure of Minimum Counterexample} \label{sub:premECE2}~~\medskip

Let  $G$ be a minimal counterexample to Theorem \ref{thm:main2} such that $G$ does not admit any $2$-distance coloring with $\D+6$ colors, but any proper subgraph of $G$ does. So $G$ is a $(\D+6)$-critical graph. Obviously,  $G$ is connected and $\delta(G)\geq 2$.  
Similarly as Section \ref{sec:premECE}, we abbreviate $6$-light, $6$-heavy, $6$-expendable vertices to light, heavy and expendable vertices throughout this section.

We  begin with outlining some structural properties that $G$ must carry. The following can be easily obtained from Lemma \ref{lem:light-heavy} and Corollary  \ref{cor:Dplus1}.

\begin{cor}\label{cor:properties2}
\begin{itemize}
\item[$(a)$] $G$ has no adjacent $2$-vertices.
\item[$(b)$] Every $2$-vertex having a $k$-neighbour for $k\leq 5$ is light.
\item[$(c)$] If a $3$-vertex has a $2$-neighbour, then its other neighbours are $5^+$-vertices.
\item[$(d)$] A $4$-vertex cannot have three $2$-neighbours.
\item[$(e)$] A $5$-vertex cannot have four $2$-neighbours.
\item[$(f)$] If a vertex $v$ has a light neighbour, then $D(v)\geq \D+7$.
\end{itemize}
\end{cor}

As $k=6$, we may rewrite Lemma \ref{lem:7-heavy} as follows.

\begin{cor}\label{cor:7-heavy2}
Every $6^-$-vertex having a $2$-neighbour is heavy. In particular, every $5$-vertex having a $3(1)$-neighbour is heavy as well.
\end{cor}

\begin{prop}\label{prop:3-4-52} 
\begin{itemize}
\item[$(a)$] A $3(1)$-vertex cannot be adjacent to a $5^-$-vertex having a $2$-neighbour.
\item[$(b)$] A $4(2)$-vertex is adjacent to neither a $5^-$-vertex having two $2$-neighbours nor a $4^-$-vertex having a $2$-neighbour.
\end{itemize}
\end{prop}

\begin{proof}
We omit  the proof since it is similar to Proposition \ref{prop:3-4-5}.
\end{proof}

We call a path $xyz$ as a \emph{poor path} if $6\leq d(y) \leq 9$, $2\leq d(x) \leq 3$ and $2\leq d(z) \leq 3$. If a poor path $xyz$ lies on the boundary of a face $f$, then the vertex $y$ is called \emph{$f$-poor vertex}.
It follows from the definition of the poor path that we can bound the number of those paths in a face.

\begin{cor}\label{cor:poor-path2}
Each face $f$ has at most  $ \floor[\big]{\frac{\ell(f)}{2}}$  $f$-poor vertices.  
\end{cor}

In the rest of the paper, we will apply discharging to show that $G$ does not exist. We assign to each vertex $v$ a charge $\mu(v)=\frac{3d(v)}{2}-5$ and to each face $f$ a charge $\mu(f)=\ell(f)-5$. By Euler's formula, we have
$$\sum_{v\in V}\left(\frac{3d(v)}{2}-5\right)+\sum_{f\in F}(\ell(f)-5)=-10$$

We next present some rules and redistribute accordingly. Once the discharging finishes, we check the final charge $\mu^*(v)$ and $\mu^*(f)$. If $\mu^*(v)\geq 0$ and $\mu^*(f)\geq 0$, we get a contradiction that no such a counterexample can exist.

\subsection{Discharging Rules} \label{sub:2} ~\medskip

We apply the following discharging rules.

\begin{itemize}

\item[\textbf{R1:}] Every $2$-vertex receives $1$  from each of its neighbour.
\item[\textbf{R2:}] Every $3(1)$-vertex receives $\frac{1}{2}$  from each of its $5$-neighbour; $\frac{2}{3}$  from each of its $6$-neighbour; $\frac{3}{4}$  from each of its $(7\text{-} 8)$-neighbour; $\frac{5}{6}$  from each of its $9$-neighbour.
\item[\textbf{R3:}] Every light $3(0)$-vertex receives $\frac{1}{6}$  from each of its $9^-$-neighbour.
\item[\textbf{R4:}] Let $v$ be a heavy $3(0)$-vertex.
\begin{itemize}
\item[$(a)$] If $v$ is adjacent to a light $3$-vertex, then $v$ receives $\frac{1}{6}$  from each $5$-neighbour; $\frac{1}{3}$  from each $(6\text{-} 8)$-neighbour, $\frac{1}{2}$  from each $9$-neighbour.
\item[$(b)$] If $v$ is not adjacent to a light $3$-vertex, then $v$ receives  $\frac{1}{6}$  from each $5$-neighbour; $\frac{1}{4}$  from each $6$-neighbour; $\frac{1}{3}$  from each $7$-neighbour; $\frac{1}{2}$  from each $(8\text{-} 9)$-neighbour.
\end{itemize}
\item[\textbf{R5:}] Let $v$ be a $4$-vertex.
\begin{itemize}
\item[$(a)$] If $v$ is a $4(1)$-vertex adjacent to exactly one $3$-vertex, then $v$ receives  $\frac{1}{12}$  from each $6$-neighbour;  
$\frac{1}{6}$  from each $(7\text{-} 8)$-neighbour.
\item[$(b)$] If $v$ is a $4(2)$-vertex, then $v$ receives  $\frac{1}{2}$  from each $(5\text{-} 8)$-neighbour.
\end{itemize}

\item[\textbf{R6:}] Every $5(3)$-vertex receives $\frac{1}{2}$  from each $(7\text{-} 8)$-neighbour.

\item[\textbf{R7:}] Every $9$-vertex gives $\frac{1}{2}$  to each of its $k$-neighbour for $4\leq k \leq 6$.
\item[\textbf{R8:}] Every $10^+$-vertex gives $1$ to each of its $8^-$-neighbour.

\item[\textbf{R9:}] If a $7$-vertex $v$ is  adjacent to two $7^+$-vertices $u$ and $w$, then $v$ gives $\frac{1}{8}$ to each  face containing one of  $uv$, $wv$.

\item[\textbf{R10:}] If  an $8$-vertex $v$ is  adjacent to an $8$-vertex $u$, then $v$ gives $\frac{1}{8}$ to each  face containing $uv$.

\item[\textbf{R11:}] If a $9$-vertex $v$ is adjacent to a $7^+$-vertex $u$, then $v$ gives $\frac{1}{4}$ to each face containing $uv$. 

\item[\textbf{R12:}] If a $10^+$-vertex $v$ is adjacent to a $9^+$-vertex $u$, then $v$ gives $\frac{1}{2}$ to each face containing $uv$. 

\item[\textbf{R13:}] Every face $f$ transfers its positive charge equally to  its incident $f$-poor vertices.

\end{itemize}

\vspace*{2em}

\noindent
\textbf{Checking} $\mu^*(v), \mu^*(f)\geq 0$, for $v\in V(G), f\in F(G)$\medskip

Clearly $\mu^*(f)\geq 0$ for each $f\in F(G)$, since every face transfers its positive charge equally to its incident $f$-poor vertices by R10. \medskip

We pick a vertex $v\in V(G)$ with $d(v)=k$. \medskip

\textbf{(1).} Let $k=2$.  By Corollary \ref{cor:properties2}(a), $v$ is adjacent to two $3^+$-vertices. It follows from applying R1 that $v$ receives $1$  from each of its neighbours, and so  $\mu^*(v)\geq 0$. \medskip 

\textbf{(2).} Let $k=3$. By Corollary \ref{cor:properties2}(c), $v$ has at most one $2$-neighbour. 
First, suppose that $v$ has a $2$-neighbour $u$. Note that the other neighbours of $v$ must be $5^+$-vertices by Corollary \ref{cor:properties2}(c). In fact, if $z_1$ and $z_2$ are the neighbours of $v$ other than $u$, then $d(z_1)+d(z_2)\geq 15$ since $v$ is heavy by Corollary \ref{cor:7-heavy2}. 
We then conclude that either $z_1$ and $z_2$ are $6^+$-vertices or one of $z_1,z_2$ is a $10^+$-vertex. 
It follows from applying R2 and R8 that $v$ receives totally at least $\frac{3}{2}$ from its $5^+$-neighbours. Hence, we have $\mu^*(v)\geq -\frac{1}{2}+\frac{3}{2}- 1= 0$ after $v$ sends $1$ to $u$ by R1.

Let us now assume that $v$ has no $2$-neighbour, i.e, $v$ is a $3(0)$-vertex. If $v$ is a light vertex, then each neighbour of $v$ is heavy by Lemma \ref{lem:light-heavy}. Thus $v$ receives at least $\frac{1}{6}$ from each of it neighbours by R3 and R8, and so $\mu^*(v)\geq -\frac{1}{2}+3\times\frac{1}{6}= 0$. We therefore suppose that $v$ is heavy. If $v$ has a light $3$-neighbour, then the other neighbours of $v$, say $z_1$ and $z_2$, must be $4^+$-vertices by Corollary \ref{cor:properties2}(f). In particular, we have $d(z_1)+d(z_2)\geq 14$.   
It follows from applying R4(a) and R8 that $v$ receives totally at least $\frac{2}{3}$ from its $5^+$-neighbours. Thus, $\mu^*(v)\geq -\frac{1}{2}+\frac{2}{3}- \frac{1}{6}= 0$ after $v$ sends $\frac{1}{6}$ to its light $3$-neighbour by R3.  
Suppose now that $v$ has no light $3$-neighbour. Denote by $z_1,z_2,z_3$ the neighbours of $v$. Since $v$ is heavy, we have $d(z_1)+d(z_2)+d(z_3)\geq \D+6\geq 16$. By applying R4(b) and R8, we deduce that
$v$ receives totally at least $\frac{1}{2}$ from its $5^+$-neighbours. Hence, $\mu^*(v)\geq -\frac{1}{2}+\frac{1}{2} = 0$. \medskip

\textbf{(3).} Let $k=4$. The initial charge of $v$ is $\mu(v)=\frac{3d(v)}{2}-5=1$.  If $v$ has no $2$-neighbour, then $\mu^*(v)\geq 1-3\times\frac{1}{6}> 0$ after $v$ sends $\frac{1}{6}$ to each of its light $3(0)$-neighbours by R3. Therefore, we assume that $v$ has at least one $2$-neighbour. In fact, $v$ can have at most two $2$-neighbours by Corollary \ref{cor:properties2}(d). Thus $1\leq n_2(v)=t \leq 2$. Denote by $x_1,x_2,\ldots,x_t$ the $2$-neighbours of $v$. Note that $v$ is a heavy vertex by Corollary \ref{cor:7-heavy2}. Since the $2$-neighbours of $v$ are light vertices by Corollary \ref{cor:properties2}(b), we have  $D(v)\geq \D+7\geq 17$ by Corollary \ref{cor:properties2}(f).
Besides, $v$  cannot have any $3(1)$-neighbour by Proposition \ref{prop:3-4-52}(a).\medskip

Let $n_2(v) = 1$. Observe that $v$ has at most two $3$-neighbours, since $v$ is heavy. Also, $v$ gives charge only to its $2$- and $3$-neighbours by R1 and R3. By this way, if $v$ has no $3$-neighbour, then  $\mu^*(v)\geq 1-1= 0$ after $v$ transfers $1$  to $x_1$ by R1.  
If $v$ has exactly one $3$-neighbour $u$, then  $v$ has either two $6$-neighbours or at least one $7^+$-neighbour, as  $D(v)\geq 17$. So, $v$ receives totally at least $\frac{1}{6}$ from its $6^+$-neighbours by R5(a).  Hence, we have $\mu^*(v)\geq 1+2\times\frac{1}{12}-1-\frac{1}{6}=0$ after $v$ sends $1$ to $x_1$ by R1;  $\frac{1}{6}$ to $u$ by R3.
If $v$ has two $3$-neighbours $u$ and $w$, then the last neighbour of $v$ must be $9^+$-vertex, as  $D(v)\geq 17$. In such a case, $v$ receives at least $\frac{1}{2}$ from its $9^+$-neighbour by R7-R8. Hence $\mu^*(v)\geq 1+\frac{1}{2}-1-2\times\frac{1}{6}> 0$ after $v$ sends $1$ to $x_1$ by R1 and  $\frac{1}{6}$  to each of  $u,w$ by R3. \medskip  

Let $n_2(v) = 2$. Denote by $u$ and $w$ the neighbour of $v$ other than $x_1$ and $x_2$. Since $v$ is heavy, and has two expendable $2$-neighbours, we have  $D(v)\geq\D+6+2 \geq 18$ by Lemma \ref{lem:DplusS}, which implies that  each of $u,w$ is a $4^+$-vertex. In fact, either both $u$ and $w$ are $5^+$-vertices or at least one of $u,w$ is a $10^+$-vertex. 
In each cases, $v$ receives totally at least $1$ from $u$ and $w$ by applying R5(b) and R8.  Thus,  $\mu^*(v)\geq1-2\times\frac{1}{2}-2\times 1 =0$ after $v$ sends $1$ to each of $x_1,x_2$ by R1.  \medskip

\textbf{(4).} Let $k=5$. We have $\mu(v)=\frac{3d(v)}{2}-5=\frac{5}{2}$. Notice that if $v$ has no $2$-neighbour, then $v$ gives at most $\frac{1}{2}$ to each of its neighbours by R2-R5, so $\mu^*(v)\geq \frac{5}{2}-5\times \frac{1}{2}=0$. Thus, we may assume that $1\leq n_2(v)=t \leq 3$ where the last inequality comes from Corollary \ref{cor:properties2}(e). Denote by $x_1,x_2,\ldots,x_t$ the $2$-neighbours of $v$. Note that $v$ is a heavy vertex by Corollary \ref{cor:7-heavy2}. Since the $2$-neighbours of $v$ are light vertices by Corollary \ref{cor:properties2}(b), we have  $D(v)\geq \D+7\geq 17$ by Corollary \ref{cor:properties2}(f). \medskip

Let $n_2(v) = 1$. Observe that $v$ has a neighbour $u$ different from $2$-, $3$- and $4(2)$-vertices by Lemma \ref{lem:DplusS}, since $v$ is heavy. Clearly, $u$ does not receive any charge from $v$. Thus, $\mu^*(v)\geq \frac{5}{2}-1-3\times \frac{1}{2}=0$ after $v$ sends $1$ to $x_1$ by R1;  at most $\frac{1}{2}$  to each of its neighbours other than $x_1$ and $u$ by R2-R5. \medskip

Let $n_2(v) = 2$. Clearly, no neighbour of $v$ can be a $3(1)$- or $4(2)$-vertex by Proposition \ref{prop:3-4-52}(a)-(b).  Moreover, $v$ has at most two $3(0)$-neighbours, since $v$ is heavy. This implies that $v$ has a neighbour which does not receive any charge from $v$.  It then follows  $\mu^*(v)\geq \frac{5}{2}-2\times 1-2\times \frac{1}{6}>0$ after $v$ sends $1$ to each of $x_1,x_2$ by R1 and at most $\frac{1}{6}$ to each of its $3$-neighbours by R3-R4.  \medskip 

Let $n_2(v) = 3$.  Denote by $z_1$ and $z_2$ the neighbours of $v$ other than  $x_1,x_2,x_3$.  Clearly, none of $z_i$'s can be a $3(1)$- or $4(2)$-vertex by Proposition \ref{prop:3-4-52}(a)-(b). If one of $z_i$'s is a $3(0)$-vertex, then $v$ would have a $10^+$-neighbour by Lemma \ref{lem:DplusS}, since $v$ is heavy, and so $v$ receives $1$ from its $10^+$-neighbour by R8. Thus,  $\mu^*(v)\geq \frac{5}{2}+1-3\times 1- 2\times\frac{1}{6}>0$  after $v$ transfers $1$  to each $x_i$ by R1 and $\frac{1}{6}$ to each of its $3(0)$-neighbours by R3-R4.  If $v$  has no $3$-neighbour, then at least one of $z_i$'s should be $7^+$-vertex by Lemma \ref{lem:DplusS}, and so $v$ receives at least $\frac{1}{2}$ from its $7^+$-neighbour by R6-R8.  It then follows that  $\mu^*(v)\geq \frac{5}{2}+\frac{1}{2}-3\times 1=0$  after $v$ transfers $1$  to each $x_i$ by R1.\medskip

\textbf{(5).} Let $k=6 $. The initial charge of $v$ is $\mu(v)=\frac{3d(v)}{2}-5=4$.  We first observe that if $v$ has no $2$-neighbour, then  $\mu^*(v)\geq 4-6\times \frac{2}{3}= 0$ after $v$ sends at most $\frac{2}{3}$ to each of its neighbour by R2-R5. Thus we may assume that $v$ has at least one $2$-neighbour. It then follows from Corollary \ref{cor:7-heavy2} that $v$ is a heavy vertex, and this infers 
$1\leq n_2(v)=t \leq 5$. Denote by $x_1,x_2,\ldots,x_t$ the $2$-neighbours of $v$, and let $y_i$ be the other neighbour of $x_i$ different from $v$. Let $f_1,f_2\ldots,f_6$ be faces incident to $v$, and suppose that they have a clockwise order on the plane.\medskip

Let $n_2(v)=1$. Recall that $v$ is heavy. Therefore, if $v$ has four $3(1)$-neighbours, then the last neighbour of $v$ must be $6^+$-vertex, which does not receive any charge from $v$. Thus $\mu^*(v)\geq 4-1-4\times \frac{2}{3}> 0$ after $v$ sends  $1$  to $x_1$ by R1; at most $\frac{2}{3}$ to each of its $3(1)$-neighbours by R2. If $v$ has three $3(1)$-neighbours, then $v$ has either a $4(0)$- or a $5^+$-neighbour, and neither receives any charge from $v$. Hence, $\mu^*(v)\geq 4-1-4\times \frac{2}{3}> 0$ after $v$ sends  $1$  to $x_1$ by R1; at most $\frac{2}{3}$ to each of its $3^+$-neighbour other than $4(0)$- and $5^+$-vertices by R2-R5. Finally, if $v$ has at most two $3(1)$-neighbours, then $\mu^*(v)\geq 4-1-2\times \frac{2}{3}-3\times \frac{1}{2}> 0$ after $v$ sends  $1$  to $x_1$ by R1; at most $\frac{2}{3}$ to each of its $3(1)$-neighbours by R2; at most $\frac{1}{2}$ to each of its other neighbours by R3-R5. \medskip

Let $n_2(v)=2$. Similarly as above,  if $v$ has three $3(1)$-neighbours, then the last neighbour of $v$ must be a $6^+$-vertex. Thus $\mu^*(v)\geq 4-2\times 1-3\times \frac{2}{3}= 0$ after $v$ sends  $1$  to each $x_i$ by R1; at most $\frac{2}{3}$ to each of its $3(1)$-neighbours by R2. If $v$ has two $3(1)$-neighbours, then $v$ has either a $4(0)$- or a  $5^+$-neighbour, and neither receives any charge from $v$. Hence, $\mu^*(v)\geq 4-2\times 1-3\times \frac{2}{3}= 0$ after $v$ sends  $1$  to each $x_i$ by R1; at most $\frac{2}{3}$ to each of its $3^+$-neighbour other than $4(0)$- and $5^+$-vertices by R2-R5.
Finally, suppose that $v$ has at most one $3(1)$-neighbour. Then $v$ must have a neighbour different from $2$-, $3(1)$- and $4(2)$-vertex, which receives at most $\frac{1}{3}$ from $v$ by R3-R5. Thus $\mu^*(v)\geq 4-2\times 1-\frac{2}{3}-2\times\frac{1}{2}-\frac{1}{3}= 0$ after $v$ sends  $1$  to each $x_i$ by R1; $\frac{2}{3}$ to its $3(1)$-neighbour by R2; $\frac{1}{2}$ to each of its $4(2)$-neighbour by R5(b); at most $\frac{1}{3}$ to each of its other neighbours by R3-R5. \medskip

Let $n_2(v)=3$. First we assume that $v$ has a $3(1)$-neighbour $u$. 
If one of $y_i$'s is a $9^-$-vertex, then $x_i$ would be an expendable (also light) vertex. In such a case,  all the other $x_i$'s would be light vertices as well. 
However, this implies that $D(u)\leq \D+4+|S|$, where $S$ is the set of light vertices of distance at most $2$ from $u$, a contradiction by Lemma \ref{lem:DplusS} and Corollary \ref{cor:7-heavy2}. 
Thus we may assume that all $y_i$'s are $10^+$-vertices. 
Note that the neighbour of $u$ different from $v$ and $2$-vertex is a $9^+$-vertex as $u$ is heavy by Corollary \ref{cor:7-heavy2}.  Since $v$ is weak-adjacent to three $10^+$-vertices, and star-adjacent to a $9^+$-vertex, we conclude that there exists a face $f_{i}$ for $i\in [6]$ such that it is either a $5$-face having two adjacent vertices $x$ and $y$ with $d(x)\geq 9$ and $d(y)\geq 10$ or a $6^+$-face. 
If $f_i$ has such vertices $x,y$, then $f_i$ gets totally $\frac{3}{4}$ from $x,y$ by R11-R12, and so $f_i$ sends it to $v$ by R13. On the other hand, if $f_i$ is a $6$-face, then $f_i$ transfers at least $\frac{1}{2}$ to $v$ by applying R13 together with Corollary \ref{cor:poor-path2}.
Consequently, $v$ receives at least $\frac{1}{2}$ from $f_i$.
Recall that if $v$ has two $3(1)$-neighbours, then the last neighbour of $v$ is a $6^+$-vertex, since $v$ is heavy. Moreover, $v$ has a neighbour different from $2$-, $3(1)$- and $4(2)$-vertex, which receives at most $\frac{1}{3}$ from $v$ by R3-R5.
Hence, if $v$ has two $3(1)$-neighbours, then $\mu^*(v)\geq 4+\frac{1}{2}-3\times 1-2\times \frac{2}{3}> 0$  after  $v$ transfers $1$  to each $x_i$ by R1;  $\frac{2}{3}$ to each of its $3(1)$-neighbour  by R2. If $v$ has at most one $3(1)$-neighbour, then $\mu^*(v)\geq 4+\frac{1}{2}-3\times 1- \frac{2}{3}-\frac{1}{2}-\frac{1}{3}= 0$  after  $v$ transfers $1$  to each $x_i$ by R1;  $\frac{2}{3}$ to each of its $3(1)$-neighbour  by R2; $\frac{1}{2}$ to each of its $4(2)$-neighbour by R5(b); at most $\frac{1}{3}$ to each of its other neighbours by R3-R5. 

Now we assume that $v$ has no $3(1)$-neighbour. Clearly, $v$ has a $6^+$-neighbour when  $v$ has two $4(2)$-neighbours, since $v$ is heavy. Moreover,  if $v$ has at most one $4(2)$-neighbour, then $v$ would have either a $4(0)$- or  a $5^+$-neighbour, and neither receives any charge from $v$. Hence, we have $\mu^*(v)\geq 4-3\times 1- 2\times \frac{1}{2}= 0$ or $\mu^*(v)\geq 4-3\times 1-  \frac{1}{2}-\frac{1}{3}> 0$ after $v$ sends  $1$  to each $x_i$ by R1; at most $\frac{1}{2}$ to each of its $4(2)$-neighbours by R5(b); at most $\frac{1}{3}$ to each of its other neighbour by R3-R5. \medskip

Let $n_2(v)\geq 4$.   
Similarly as above, if one of $y_i$'s is a $9^-$-vertex, then $x_i$ would be expendable (also light) vertex, and so all the other $x_i$'s would be light vertices as well. It follows from Lemma \ref{lem:DplusS} that $D(v)\geq \D+6+|S|$, where $S$ is the set of light vertices of distance at most $2$ from $v$. Then $v$ must have a $6^+$-neighbour, and $v$ cannot have any $3(1)$- or $4(2)$-neighbour by Corollary \ref{cor:7-heavy2}. In particular, $v$ has a $9^+$-neighbour when $v$ has a $3(0)$- or $4(1)$-neighbour. By R7-R8, $v$ receives at least $\frac{1}{2}$ from its $9^+$-neighbours. Thus we have $\mu^*(v)\geq 4-4\times 1> 0$ or $\mu^*(v)\geq 4-4\times 1- \frac{1}{3}+ \frac{1}{2}> 0$ after $v$ sends  $1$  to each $x_i$ by R1; at most $\frac{1}{3}$ to each of its $3(0)$- and $4(1)$-neighbours by R3-R5.

We may further assume that all $y_i$'s are $10^+$-vertices. 
Since $v$ is weak adjacent to four $10^+$-vertex, we conclude that there exist $f_{i_1},f_{i_2}$ for $i_1,i_2\in [6]$ such that it is either a $5$-face having two adjacent $10^+$-vertices $x$ and $y$ or a $6^+$-face. It follows that each of those faces transfers at least $1$ to $v$ by applying R12 and R13 together with Corollary \ref{cor:poor-path2}. Hence, $\mu^*(v)\geq 4+2-6\times 1= 0$ after  $v$ transfers at most $1$  to each of its neighbour by R1-R5. \medskip

\textbf{(6).} Let $k=7$. The initial charge of $v$ is $\mu(v)=\frac{3d(v)}{2}-5=\frac{11}{2}$. Note that if $v$ has two $7^+$-neighbours $u$ and $w$, then $v$ sends $\frac{1}{8}$ to each face containing one of $vu$, $vw$ by R9. Consequently, it follows that $\mu^*(v)\geq \frac{11}{2}-5\times 1-4\times \frac{1}{8}= 0$ after  $v$ transfers $1$  to each of its neighbours other than $u,w$ by R1-R6. Thus we further assume that $v$ has at most one $7^+$-neighbour, i.e., the rule R9 cannot be applied to $v$. Remark that any $3^+$-vertex  receives at most $\frac{3}{4}$ from its $7$-neighbour by R2-R6. Thus  $\mu^*(v)\geq \frac{11}{2}-1-6\times \frac{3}{4}= 0$ when  $v$ has at most one $2$-neighbour.  We may therefore assume that $2\leq n_2(v)=t \leq 7$. Denote by $x_1,x_2,\ldots,x_t$ the $2$-neighbours of $v$, and let $y_i$ be the other neighbour of $x_i$ different from $v$.  Let $z_1,z_2,\ldots,z_{7-t}$ be the neighbour of $v$ other than $x_1,x_2,\ldots,x_t$.   Let $f_1,f_2\ldots,f_7$ be faces incident to $v$, and suppose that they have a clockwise order on the plane. Observe that any vertex different from $2$- and $3(1)$-vertex receives at most $\frac{1}{2}$ from its $7$-neighbour by R3-R6.\medskip

Let $n_2(v)=2$. If $v$ has a neighbour other than $2$- and $3(1)$-vertex, then   $\mu^*(v)\geq \frac{11}{2}-2\times 1-4\times \frac{3}{4}-\frac{1}{2}= 0$ after  $v$ transfers $1$  to each $x_i$ by R1;   $\frac{3}{4}$ to each of its $3(1)$-neighbours by R2; at most $\frac{1}{2}$  to each of its other neighbours by R3-R6.        
We may therefore  assume that all neighbours of $v$ other than $x_i$'s are $3(1)$-vertices, i.e., $v$ has five $3(1)$-neighbours. In this case, $v$ is a light vertex, since $D(v)=19 < \D+6+e_6(v)$ with $e_6(v)\geq 5$. So, all $x_i$'s must be heavy by Lemma \ref{lem:light-heavy}. This means that each $y_i$ is a $10^+$-vertex by Corollary \ref{cor:Dplus1}. 
On the other hand, the neighbour of each $z_i$ different from $v$ and $2$-vertex is a $9^+$-vertex by Corollary \ref{cor:7-heavy2} together with Lemma \ref{lem:DplusS}. Thus $v$ is star-adjacent to five $9^+$-vertices. Since $v$ is weak-adjacent to two $10^+$-vertices and star-adjacent to five $9^+$-vertex, we conclude that there exist a face $f_{i}$ for $i\in [7]$ such that it is either a $5$-face having two adjacent $9^+$-vertices $x$ and $y$ or a $6^+$-face.
If $f_i$ has such two vertices $x,y$, then it gets totally $\frac{1}{2}$ from $x$ and $y$ by R11-R12. It then follows from applying R13 together with Corollary \ref{cor:poor-path2} that $f_i$ transfers at least $\frac{1}{2}$ to $v$.
Hence,  $\mu^*(v)\geq \frac{11}{2}+\frac{1}{2}-2\times 1-5\times \frac{3}{4}> 0$ after  $v$ transfers $1$  to each $x_i$ by R1;  $\frac{3}{4}$ to each of its $3(1)$-neighbour  by R2. \medskip

Let $n_2(v)=3$. If $v$ has at most two $3(1)$-neighbours, then  $\mu^*(v)\geq \frac{11}{2}-3\times 1-2\times \frac{3}{4}-2\times \frac{1}{2}= 0$ after  $v$ transfers $1$  to each $x_i$ by R1;   $\frac{3}{4}$ to each of its $3(1)$-neighbour by R2; at most $\frac{1}{2}$  to each of its other neighbours by R3-R6.  
We now suppose that $v$ has at least three $3(1)$-neighbours. If $v$ has also a $4^+$-neighbour $u$ different from a $4(2)$- and $5(3)$-vertex, then, similarly as above, we have  $\mu^*(v)\geq \frac{11}{2}-3\times 1-3\times \frac{3}{4}- \frac{1}{6}> 0$ after  $v$ transfers $1$  to each $x_i$ by R1;   $\frac{3}{4}$ to each of its $3(1)$-neighbour by R2; at most $\frac{1}{6}$  to $u$ by R5(a). If $v$ has no such $4^+$-neighbour, then $v$ would be a light vertex since $D(v) < \D+6+e_6(v)$. Thus, each $y_i$ must be $10^+$-vertex by Lemma \ref{lem:light-heavy} and Corollary \ref{cor:Dplus1}. Moreover, similarly as above,  $v$ is star-adjacent to three $9^+$-vertices. 
Since $v$ is weak-adjacent to three $10^+$-vertices and star-adjacent to three $9^+$-vertex, we conclude that there exist $f_{i}$ for $i\in [7]$ such that it is either a $5$-face having two adjacent vertices $x$ and $y$ with $d(x)\geq 9$ and $d(y)\geq 10$ or a $6^+$-face. Similarly as above, $f_i$ transfers at least $\frac{1}{2}$ to $v$ by applying R11-R13 together with Corollary \ref{cor:poor-path2}.
 Hence,  $\mu^*(v)\geq \frac{11}{2}+\frac{1}{2}-3\times 1-4\times \frac{3}{4}= 0$  after  $v$ transfers $1$  to each $x_i$ by R1;  at most $\frac{3}{4}$ to each of its other neighbours  by R2-R6.  \medskip

Let $n_2(v)=4$.  Clearly, if $v$ has no $3(1)$-neighbour, then   $\mu^*(v)\geq \frac{11}{2}-4\times 1-3\times \frac{1}{2}= 0$ after  $v$ transfers $1$  to each $x_i$ by R1;  at most $\frac{1}{2}$  to each of its other neighbours by R3-R6. We may therefore assume that $v$ has a $3(1)$-neighbour, say $z_1$. Recall that $z_2$ and $z_3$ are the neighbours of $v$ other than $x_i$'s and $u$. If one of $z_2,z_3$, say $z_2$, is a $4^+$-vertex other than $4(1)$, $4(2)$ and $5(3)$, then   $\mu^*(v)\geq \frac{11}{2}-4\times 1-2\times \frac{3}{4}= 0$ after  $v$ transfers $1$  to each $x_i$ by R1;  at most $\frac{3}{4}$ to each of its neighbours other than $x_i$'s and $z_2$ by R2-R6. On the other hand, if each of $z_2,z_3$ is a $3(0)$- or $4(1)$-vertex, then $\mu^*(v)\geq \frac{11}{2}-4\times 1-\frac{3}{4}-2\times \frac{1}{3}= 0$ after  $v$ transfers $1$  to each $x_i$ by R1;  $\frac{3}{4}$ to its $3(1)$-neighbour $z_1$ by R2; at most $\frac{1}{3}$ to each of $z_2,z_3$ by R3-R5.
We may therefore suppose that  one of $z_1,z_2$ is a $3(1)$-, $4(2)$- or $5(3)$-vertex while the other is a $3(0)$, $3(1)$, $4(1)$, $4(2)$- or $5(3)$-vertex. This means that $v$ is a light vertex, since $D(v)<\D+6+e_6(v)$. It then follows from Lemma \ref{lem:light-heavy} that each $x_i$ is heavy, and so each $y_i$ must be $10^+$-vertex by Corollary \ref{cor:Dplus1}. 
Since $v$ is weak-adjacent to four $10^+$-vertices, we conclude that there exist a face $f_{i}$ for $i\in [7]$ such that it is either a $5$-face having two adjacent $10^+$-vertices or a $6^+$-face. Similarly as above, $f_i$ transfers at least $1$ to $v$ by applying R12 and R13 together with Corollary \ref{cor:poor-path2}. Hence,  $\mu^*(v)\geq \frac{11}{2}+1-4\times 1-3\times \frac{3}{4}> 0$  after  $v$ transfers $1$  to each $x_i$ by R1;  at most $\frac{3}{4}$ to each of its other neighbours by R2-R6.  \medskip

Let $n_2(v)=5$. If two of  $y_i$'s are $8^-$-vertices, then the corresponding vertices $x_i$'s would be expendable (also light) vertices. In this case,  all the other $x_i$'s would be light vertices as well. It follows from Lemma \ref{lem:light-heavy} that $v$ is heavy, and so $D(v)\geq \D+6+|S|$ with $|S|\geq 5$ by Lemma \ref{lem:DplusS},  where $S$ is the set of light vertices of distance at most $2$ from $v$. It can be easily seen that $v$ has a $6^+$-neighbour.  In particular, $v$ cannot have a $3(1)$-neighbour $u$ by Lemma \ref{lem:DplusS} and Corollary \ref{cor:7-heavy2}, since $D(u)\leq \D+3+|S|$ with $|S|\geq 6$,  where $S$ is the set of light vertices of distance at most $2$ from $u$.
Thus, $\mu^*(v)\geq \frac{11}{2}-5\times 1-\frac{1}{2}= 0$ after  $v$ transfers $1$  to each $x_i$ by R1;  $\frac{1}{2}$ to its neighbour other than $2$ and $6^+$-vertex by R3-R6. We may therefore assume that  at most one of $x_i$'s is an expendable vertex. 

Suppose first that only one of  $x_i$'s is expendable, say $x_1$. Then each of $x_2,x_3,x_4,x_5$ is adjacent to a $9^+$-vertex, i.e., $v$ is weak-adjacent to four $9^+$-neighbours. It follows from Lemma \ref{lem:light-heavy} that $v$ is heavy, and so $D(v)\geq \D+6+|S|$ with $|S|\geq 5$ by Lemma \ref{lem:DplusS},  where $S$ is the set of light vertices of distance at most $2$ from $v$. We then conclude that $v$ cannot have two vertices forming $3(1)$-, $4(2)$- or $5(3)$-vertices. Moreover, $v$ has no $4$-neighbour when it has a $3(1)$-neighbour. This means that $v$ sends totally at most $\frac{5}{6}=\max\{\frac{3}{4},\frac{1}{2}+\frac{1}{3}\}$ to its $3^+$-neighbours.  On the other hand, since $v$ is weak-adjacent to four $9^+$-vertices, we conclude that there exists a face $f_{i}$ for $i\in [7]$ such that it is either a $5$-face having two adjacent $9^+$-vertices or a $6^+$-face. Similarly as above, $f_i$ transfers at least $\frac{1}{3}$ to $v$ by applying R11-R13 together with Corollary \ref{cor:poor-path2}.   Thus, $\mu^*(v)\geq \frac{11}{2}+\frac{1}{3}-5\times 1-\frac{5}{6}= 0$  after  $v$ transfers $1$  to each $x_i$ by R1 and  totally at most $\frac{5}{6}$ to its $3^+$-neighbours by R2-R6.

Suppose now that none of  $x_i$'s is expendable. Then each of those $2$-vertices has a $9^+$-neighbour. Since $v$ is weak-adjacent to five $9^+$-vertices, we conclude that there exist three faces $f_{i_1},f_{i_2},f_{i_3}$ for $i_1,i_2,i_3\in [7]$ such that each of them is either a $5$-face having two adjacent $9^+$-vertices or a $6^+$-face. Recall that if $f_i$ is a $6^+$-face, then $v$ receives at least $\frac{1}{3}$ from $f_i$ by R13 together with Corollary \ref{cor:poor-path2}. Thus, similarly as above, each of $f_{i_1},f_{i_2},f_{i_3}$ transfers at least $\frac{1}{3}$ to $v$ by applying R11-R13 together with Corollary \ref{cor:poor-path2}.  Hence,  $\mu^*(v)\geq \frac{11}{2}+3\times \frac{1}{3}-5\times 1-2\times \frac{3}{4}= 0$  after  $v$ transfers $1$  to each $x_i$ by R1;  at most $\frac{3}{4}$ to each of its other neighbours by R2-R6.   \medskip

Let $n_2(v)\geq 6$. Similarly as above, if two of  $x_i$'s are expendable (also light) vertices, then all the other $x_i$'s would be light vertices as well, and this forces that $v$ has a $10^+$-neighbour by Lemmas \ref{lem:light-heavy} and  \ref{lem:DplusS}. Note that $v$ receives $1$ from its $10^+$-neighbour by R8. Thus, $\mu^*(v)\geq \frac{11}{2}+1-6\times 1> 0$ after  $v$ transfers $1$ to each $x_i$ by R1.
We may therefore assume that  at most one of $x_i$'s is expendable. Suppose first that only one of $x_i$'s is an expendable vertex, say $x_1$. Then each of $x_2,x_3,\ldots,x_6$ is adjacent to a $9^+$-vertex, i.e., $v$ is weak-adjacent to five $9^+$-vertices. It follows from Lemma \ref{lem:light-heavy} that $v$ is heavy, and so $D(v)\geq \D+6+|S|$ with $|S|\geq 5$ by Lemma \ref{lem:DplusS},  where $S$ is the set of light vertices of distance at most $2$ from $v$. We then conclude that $v$ has no $4^-$-, and $5(3)$-neighbours, i.e., $v$ sends charge to only its $2$-neighbours. On the other hand, since $v$ is weak-adjacent to at least five $9^+$-vertices, we conclude that there exist three faces $f_{i_1},f_{i_2},f_{i_3}$ for $i_1,i_2,i_3\in [7]$ such that each of them is either a $5$-face having two adjacent $9^+$-vertices or a $6^+$-face.  Similarly as above, each of $f_{i_1},f_{i_2},f_{i_3}$ transfers at least $\frac{1}{3}$ to $v$ by applying R11-R13 together with Corollary \ref{cor:poor-path2}.  Hence,  $\mu^*(v)\geq \frac{11}{2}+3\times \frac{1}{3}-6\times 1>0$  after  $v$ transfers at most $1$  to each of its $2$-neighbours by R1.

If none of  $x_i$'s is expendable, then  each of those $2$-vertices has a $9^+$-neighbour.  Since $v$ is weak-adjacent to at least six $9^+$-vertices, we conclude that there exist  five faces $f_{i_1},f_{i_2},\ldots, f_{i_5}$ for $i_1,i_2,\ldots,i_5\in [7]$ such that each of them is either a $5$-face having two adjacent $9^+$-vertices or a $6^+$-face. Similarly as above, each of $f_{i_1},f_{i_2},\ldots, f_{i_5}$ transfers at least $\frac{1}{3}$ to $v$ by applying R11-R13 together with Corollary \ref{cor:poor-path2}.  Hence,  $\mu^*(v)\geq \frac{11}{2}+5\times \frac{1}{3}-7\times 1> 0$  after  $v$ transfers at most $1$  to each of its neighbours by R1-R6.   \medskip

\textbf{(7).} Let $k=8$. The initial charge of $v$ is $\mu(v)=\frac{3d(v)}{2}-5=7$.  Observe that any $3^+$-vertex receives at most $\frac{3}{4}$ from its  $8$-neighbour by R2-R6. Thus  $\mu^*(v)\geq 7-4\times 1 -4\times \frac{3}{4}= 0$ when  $v$ has at most four $2$-neighbour.  We may therefore assume that $5\leq n_2(v)=t \leq 8$. Denote by $x_1,x_2,\ldots,x_t$ the $2$-neighbours of $v$, and let $y_i$ be the other neighbour of $x_i$ different from $v$.  Let $z_1,z_2,\ldots,z_{8-t}$ be the neighbour of $v$ other than $x_1,x_2,\ldots,x_t$.   Let $f_1,f_2\ldots,f_8$ be faces incident to $v$, and suppose that they have a clockwise order on the plane. Observe that any vertex other than $2$-vertex and $3(1)$-vertex   receives at most $\frac{1}{2}$ from its $8$-neighbour by R3-R6.
We remark that if $v$ has an $8$-neighbour $u$, then $v$ sends $\frac{1}{8}$ to each face containing $vu$ by R10. So, $v$ sends totally $\frac{1}{4}$ to both faces containing $vu$. In such a case, for simplicity, we say that $v$ sends $\frac{1}{4}$ to $u$ by R10 instead of saying that $v$ sends $\frac{1}{8}$ to each face containing $vu$ in the sequel. \medskip

We start by noting that if three of $x_i$'s are  expendable (also light) vertices, then, by Lemma \ref{lem:DplusS}, all the other $x_i$'s are light as well. In such a case, $v$ is heavy by Lemma \ref{lem:light-heavy}, and so it has a $4(0)$- or $5^+$-neighbour $u$ different from $5(3)$-vertex. 
In particular, $v$ cannot have a $3(1)$-neighbour $w$ by Lemma \ref{lem:DplusS} and Corollary \ref{cor:7-heavy2}, since $D(w)\leq \D+4+|S|$ with $|S|\geq 6$,  where $S$ is the set of light vertices of distance at most $2$ from $w$. 
On the other hand, if $n_2(v)=7$, then $u$ is a $9^+$-vertex by Lemma \ref{lem:DplusS}, which means that  $n_2(v)\leq 6$ when $v$ is adjacent to an $8$-vertex.
Consequently, if  $n_2(v)=7$, then  $\mu^*(v)\geq 7-7\times 1= 0$  after  $v$ transfers at most $1$ to each of its neighbours other than $u$ by R1-R6.  If  $n_2(v)\leq 6$, then  $\mu^*(v)\geq 7-6\times 1-2\times \frac{1}{2}= 0$ after $v$ transfers $1$  to each $x_i$ by R1;  at most $\frac{1}{2}$ to each of its neighbour other than $x_i$'s by R3-R6 and R10.    
Thus we may further assume that  at most two of $x_i$'s are expendable vertices. Moreover, if one of $x_i$'s is an expendable vertex, then $v$ must have a neighbour other than $2$-, $3(1)$- and $4(2)$-vertex by Lemma \ref{lem:DplusS}. \medskip

Let $n_2(v)=5$. If $v$ has a neighbour different from $2$- and $3(1)$-vertex, then $\mu^*(v)\geq 7-5\times 1-2\times \frac{3}{4}-\frac{1}{2}=0$  after  $v$ transfers $1$  to each $x_i$ by R1;   $\frac{3}{4}$ to each of its $3(1)$-neighbours by R2; at most $\frac{1}{2}$  to each of its other neighbours by R3-R6 and R10.     
We may therefore assume that $v$ has exactly three $3(1)$-neighbours.  Recall that if one of $x_i$'s is a light vertex, then $v$ must have a neighbour other than $2$-, $3(1)$- and $4(2)$-vertex. We then infer that  all $2$-neighbours of $v$ are heavy. 
This particularly implies that each $y_i$ is an $8^+$-vertex. 
Since $v$ is weak-adjacent to  five $8^+$-vertices, we conclude that there exist $f_{i_1},f_{i_2}$ for $i_1,i_2\in [8]$ such that each of them is either a $5$-face having two adjacent $8^+$-vertices $x,y$ or a $6^+$-face. If one of those faces has such vertices $x,y$, then it gets totally at least $\frac{1}{4}$ from $x,y$ by R10-R12. It then follows from applying R13 together with Corollary \ref{cor:poor-path2} that each of those faces transfers at least $\frac{1}{4}$ to $v$. Hence, $\mu^*(v)\geq 7+2\times \frac{1}{4}-5\times 1-3\times \frac{3}{4}>0$ after  $v$ transfers $1$  to each $x_i$ by R1; at most $\frac{3}{4}$ to each of its other neighbours  by R2-R6 and R10. \medskip

Let $n_2(v)=6$.  If $v$ has no $3(1)$-neighbour, then $\mu^*(v)\geq 7-6\times 1-2\times \frac{1}{2}=0$ after  $v$ transfers $1$  to each $x_i$ by R1;  at most $\frac{1}{2}$  to each of its other neighbours by R3-R6 and R10.  
If $v$ has two $3(1)$-neighbours,  then, similarly as above, all $2$-neighbours of $v$ must be heavy, and so each $y_i$ is an $8^+$-vertex. Since $v$ is weak-adjacent to  six $8^+$-vertices, we conclude that there exist $f_{i_1},f_{i_2},\ldots,f_{i_4}$ for $i_1,i_2,\ldots,i_4\in [8]$ such that each of them is either a $5$-face having two adjacent $8^+$-vertices or a $6^+$-face.  Similarly as above, each of $f_{i_1},f_{i_2},\ldots, f_{i_4}$ transfers at least $\frac{1}{4}$ to $v$ by applying R10-R13 together with Corollary \ref{cor:poor-path2}. Hence,  $\mu^*(v)\geq 7+4\times \frac{1}{4}-6\times 1-2\times \frac{3}{4}>0$ after  $v$ transfers $1$  to each $x_i$ by R1;  $\frac{3}{4}$ to each of its $3(1)$-neighbour  by R2. 

Finally, we assume that $v$ has exactly one $3(1)$-neighbour. If two of $x_i$'s are expendable (also light) vertices, then $v$ has a $4(0)$- or $5^+$-neighbour $u$ different from $5(3)$-vertex by Lemma \ref{lem:DplusS}. Thus, $\mu^*(v)\geq 7-6\times 1- \frac{3}{4}-\frac{1}{4}=0$ after  $v$ transfers $1$  to each $x_i$ by R1;  $\frac{3}{4}$ to its $3(1)$-neighbour  by R2; $\frac{1}{4}$ to its $8$-neighbour  by R10.  
If at most one of $x_i$'s is light, then $v$ would have at least five heavy $2$-neighbours. Clearly, each of those $2$-vertices has two $8^+$-neighbours. That is, $v$ is weak-adjacent to at least five  $8^+$-vertex. We then conclude that there exist $f_{i_1},f_{i_2}$  for $i_1,i_2\in [8]$ such that each of them is either a $5$-face having two adjacent $8^+$-vertices or a $6^+$-face. Similarly as above,  each of those faces transfers at least $\frac{1}{4}$ to $v$ by applying R10-R13 together with Corollary \ref{cor:poor-path2}. Hence,  $\mu^*(v)\geq 7+2\times \frac{1}{4}-6\times 1- \frac{3}{4}-\frac{1}{2}>0$ after  $v$ transfers $1$  to each $x_i$ by R1;  $\frac{3}{4}$ to its $3(1)$-neighbour  by R2;  at most $\frac{1}{2}$  to its neighbour other than $2$- and $3(1)$-vertex by R3-R6 and R10.\medskip

Let $n_2(v)=7$.  Suppose first that $v$ has a $3(1)$-neighbour $u$. If one of $x_i$'s
 is light, then $v$ would be heavy by Lemma \ref{lem:light-heavy}. However, we have $D(v)=17<\D+6+|S|$ with $|S|\geq 2$, where $S$ is the set of light vertices of distance at most  $2$ from $v$, a contradiction by Lemma \ref{lem:DplusS}. Thus, all $x_i$'s  are heavy. This in particular implies that each $y_i$ is an $8^+$-vertex. Since $v$ is weak-adjacent to seven $8^+$-vertices, we conclude that there exist six faces $f_{i_1},f_{i_2},\ldots,f_{i_6}$ for $i_1,i_2,\ldots,i_6\in [8]$ such that each of them is either a $5$-face having two adjacent $8^+$-vertices or a $6^+$-face.  Each of those faces transfers at least $\frac{1}{4}$ to $v$ by R10-R13 together with Corollary \ref{cor:poor-path2}. Hence, $\mu^*(v)\geq 7+6\times \frac{1}{4}-7\times 1- \frac{3}{4}>0$ after  $v$ transfers $1$  to each $x_i$ by R1;  $\frac{3}{4}$ to its $3(1)$-neighbour by R2.
We may further assume that $v$ has no $3(1)$-neighbour. 

Note that if two of $x_i$'s are light, then $v$ must have a $4(0)$- or $5^+$-neighbour $u$ different from $5(3)$-vertex by Lemma \ref{lem:DplusS}. In such a case, we conclude that there exist $f_{i_1},f_{i_2}$ for $i_1,i_2\in [8]$ such that each of them is either a $5$-face having two adjacent $8^+$-vertices or a $6^+$-face. Each of those faces transfers at least $\frac{1}{4}$ to $v$ by R10-R13 together with Corollary \ref{cor:poor-path2}. Hence, $\mu^*(v)\geq 7+2\times \frac{1}{4}-7\times 1- \frac{1}{4}>0$ after  $v$ transfers $1$  to each $x_i$ by R1;  $\frac{1}{4}$ to $u$  by R10 when $u$ is an $8$-vertex. 

Similarly as above, if at most one of $x_i$'s is light, then $v$ is weak-adjacent to at least six $8^+$-vertices. We then conclude that there exist $f_{i_1},f_{i_2},\ldots,f_{i_4}$ for $i_1,i_2,\ldots,i_4\in [8]$ such that each of them is either a $5$-face having two adjacent $8^+$-vertices or a $6^+$-face.   Each of those faces transfers at least $\frac{1}{4}$ to $v$ by R10-R13 together with Corollary \ref{cor:poor-path2}. Hence, $\mu^*(v)\geq 7+4\times \frac{1}{4}-7\times 1- \frac{1}{2}>0$ after  $v$ transfers $1$  to each $x_i$ by R1;  $\frac{1}{2}$ to its neighbour other than $x_i$'s  by R3-R6 and R10. \medskip

Let $n_2(v)=8$.  All neighbours of $v$ must be heavy, since otherwise, if one of $x_i$'s is light, then $v$ would be light as well, a contradiction by Lemma \ref{lem:light-heavy}. This implies that each $y_i$ is an $8^+$-vertex.
We then conclude that each face incident to $v$ is either a $5$-face having two adjacent $8^+$-vertices or a $6^+$-face.  Each of those faces transfers at least $\frac{1}{4}$ to $v$ by R10-R13 together with Corollary \ref{cor:poor-path2}. Hence, $\mu^*(v)\geq 7+8\times \frac{1}{4}-8\times 1>0$ after  $v$ transfers $1$  to each $x_i$ by R1. \medskip

\textbf{(8).} Let $k=9$. The initial charge of $v$ is $\mu(v)=\frac{3d(v)}{2}-5=\frac{17}{2}$.
We first note that if $v$ has a $7^+$-neighbours $u$, then  $\mu^*(v)\geq \frac{17}{2}-8\times 1- \frac{1}{2}=0$ after  $v$ transfers $1$  to each of its $2$-neighbour by R1; $\frac{1}{4}$ to each faces containing $vu$ by R11. Thus we further assume that $v$ has no $7^+$-neighbour, i.e., the rule R11 cannot be applied to $v$.
On the other hand, if $v$ has at most seven $2$-neighbours, then   $\mu^*(v)\geq \frac{17}{2}-7\times 1-2\times \frac{3}{4}=0$ after  $v$ transfers $1$  to each $x_i$ by R1; at most $\frac{3}{4}$ to each of its other neighbours  by R2-R7.  We may therefore assume that $8\leq n_2(v)=t \leq 9$. Denote by $x_1,x_2,\ldots,x_t$ the $2$-neighbours of $v$, and let $y_i$ be the other neighbour of $x_i$ different from $v$.  Let $f_1,f_2\ldots,f_9$ be faces incident to $v$, and suppose that they have a clockwise order on the plane. \medskip

Let $n_2(v)=8$. Denote by $z$ the $3^+$-neighbour of $v$. If four of $x_i$'s are   expendable (also light) vertices, then all $2$-neighbours of $v$ would be light vertices as well, by Lemma \ref{lem:DplusS}. In such a case, $v$ must be heavy by Lemma \ref{lem:light-heavy}, and so it has a $8^+$-neighbour $u$. Thus,  $\mu^*(v)\geq \frac{17}{2}-8\times 1- \frac{1}{2}=0$ after  $v$ transfers $1$  to each $x_i$ by R1;  $\frac{1}{4}$ to each face containing $vu$ by R11.   
We may therefore assume that  at most three of $x_i$'s are light vertices.

Suppose that exactly three of $x_i$'s are expendable (also light) vertices. By Lemma \ref{lem:light-heavy}, $v$ is heavy, and so $z$ would be heavy $3(0)$-vertex or $4^+$-vertex. In this case, $\mu^*(v)\geq \frac{17}{2}-8\times 1-\frac{1}{2}=0$ after  $v$ transfers at most $1$  to each of its $2$-neighbours by R1 and at most $\frac{1}{2}$ to $z$ by R3-R7.

Next, we suppose that exactly two of $x_i$'s are expendable (also light) vertices, say $x_7,x_8$. By Lemma \ref{lem:light-heavy}, $v$ is heavy. If $z$ is not a $3(1)$-vertex, then  $\mu^*(v)\geq \frac{17}{2}-8\times 1-\frac{1}{2}=0$ after  $v$ transfers at most $1$  to each of its $2$-neighbours by R1 and at most $\frac{1}{2}$ to $z$ by R3-R7. If $z$ is $3(1)$-vertex, then each of $x_i$ for $i\in [6]$ must be heavy, i.e., each $y_i$ for $i\in [6]$ is a $9^+$-vertex.  This implies that $v$ is weak-adjacent to at least six $9^+$-vertices.
We then conclude that there exist three faces incident to $v$ such that  each of them is either a $5$-face having two adjacent $9^+$-vertices or a $6^+$-face. Each of those faces transfers at least $\frac{1}{3}$ to $v$ by R11-R13 together with Corollary \ref{cor:poor-path2}. Hence, $\mu^*(v)\geq \frac{17}{2}+3\times \frac{1}{3}-9\times 1>0$ after  $v$ transfers at most $1$  to each of its neighbours by R1-R7.

We now suppose that exactly one of $x_i$'s is expendable (also light) vertex, say $x_8$. By Lemma \ref{lem:light-heavy}, $v$ is heavy.  Note that each $y_i$ for $i\in [7]$ is an $8^+$-vertex.  This implies that $v$ is weak-adjacent to seven $8^+$-vertices.
We then conclude that there exist two faces incident to $v$ such that  each of them is either a $5$-face having two adjacent $8^+$-vertices or a $6^+$-face. Each of those faces transfers at least $\frac{1}{4}$ to $v$ by R10-R13 together with Corollary \ref{cor:poor-path2}. Hence, $\mu^*(v)\geq \frac{17}{2}+2\times \frac{1}{4}-9\times 1=0$ after  $v$ transfers at most $1$  to each of its neighbours by R1-R7.

Finally suppose that none of $x_i$'s is an expendable (also light) vertices. Note that each $y_i$ for $i\in [8]$ is a $7^+$-vertex.  This implies that $v$ is weak-adjacent to eight $7^+$-vertices.
We then conclude that there exist four faces $f_{i_1},f_{i_2},\ldots,f_{i_4}$  for $i_1,i_2,\ldots,i_4\in [8]$  such that each of them is either a $5$-face having two adjacent $7^+$-vertices  
or a $6^+$-face. By applying R9-R13 together with   Corollary \ref{cor:poor-path2}, we infer that each of $f_{i_2},f_{i_3}$ sends at least $\frac{1}{4}$ to $v$.  Hence, $\mu^*(v)\geq \frac{17}{2}+ 2\times \frac{1}{4}-9\times 1=0$ after  $v$ transfers at most $1$  to each of its neighbours by R1-R7.   \medskip

Let $n_2(v)=9$. Observe that at most two of $x_i$'s are   expendable (also light) vertices, since otherwise, $v$ and its a neighbour would be light vertices, a contradiction by Lemma \ref{lem:light-heavy}. 

First we suppose that exactly two of $x_i$'s are expendable (also light) vertices. By Lemma \ref{lem:light-heavy}, $v$ is heavy, and so seven of $x_i$'s are heavy by Lemma \ref{lem:DplusS}, say $x_1,x_2,\ldots,x_7$. Note that each $y_i$ for $i\in [7]$ is a $9^+$-vertex.  This implies that $v$ is weak-adjacent to at least seven $9^+$-vertices.
We then conclude that there exist three faces incident to $v$ such that  each of them is either a $5$-face having two adjacent $9^+$-vertices or a $6^+$-face. Each of those faces transfers at least $\frac{1}{3}$ to $v$ by R11-R13 together with Corollary \ref{cor:poor-path2}. Hence, $\mu^*(v)\geq \frac{17}{2}+3\times \frac{1}{3}-9\times 1>0$ after  $v$ transfers at most $1$  to each of its neighbours by R1-R7.

We now suppose that exactly one of $x_i$'s is expendable (also light) vertex. By Lemma \ref{lem:light-heavy}, $v$ is heavy, and so eight of $x_i$'s are heavy by Lemma \ref{lem:DplusS}, say $x_1,x_2,\ldots,x_8$. Note that each $y_i$ for $i\in [8]$ is a $8^+$-vertex.  This implies that $v$ is weak-adjacent to at least seven $8^+$-vertices.
We then conclude that there exist two faces incident to $v$ such that  each of them is either a $5$-face having two adjacent $8^+$-vertices or a $6^+$-face. Each of those faces transfers at least $\frac{1}{4}$ to $v$ by R10-R13 together with Corollary \ref{cor:poor-path2}. Hence, $\mu^*(v)\geq \frac{17}{2}+2\times \frac{1}{4}-9\times 1=0$ after  $v$ transfers at most $1$  to each of its neighbours by R1-R7.

Now we suppose that none of $x_i$'s is an expendable (also light) vertices. Note that each $y_i$ for $i\in [9]$ is a $7^+$-vertex.  This implies that $v$ is weak-adjacent to at least nine $7^+$-vertices.
We then conclude each face incident to $v$ sends at least $\frac{1}{4}$ to $v$  by applying R9-R13 together with   Corollary \ref{cor:poor-path2}. Hence, $\mu^*(v)\geq \frac{17}{2}+ 9\times \frac{1}{4}-9\times 1>0$ after  $v$ transfers at most $1$  to each of its neighbours by R1-R7.   \medskip

\textbf{(9).} Let $k\geq 10 $. By R8, $v$ sends $1$ to each of its $8^-$-neighbour.   
 If $v$ has a $9^+$-neighbour $u$, then $v$ sends $\frac{1}{2}$ to each face containing $vu$ by R12. Consequently, $\mu^*(v)\geq 0$.


%
%
%

\section*{Acknowledgments}
This research was supported by TUBITAK (The Scientific and Technological Research Council of Turkey) under the project number 122F250. We thank anonymous referees for carefully reading and for their invaluable comments and suggestions which have improved the presentation of this paper.

\bibliographystyle{abbrv} 
\bibliography{zd-mybib}

\begin{thebibliography}{10}

\bibitem{bu-shang}
Yuehua Bu and Chunhui Shang.
\newblock List 2-distance coloring of planar graphs without short cycles.
\newblock {\em Discrete Mathematics, Algorithms and Applications},
  8(01):1650013, 2016.

\bibitem{bu-zu}
Yuehua Bu and Xubo Zhu.
\newblock An optimal square coloring of planar graphs.
\newblock {\em Journal of Combinatorial Optimization}, 24(4):580--592, 2012.

\bibitem{dong-lin-2016}
Wei Dong and Wensong Lin.
\newblock An improved bound on 2-distance coloring plane graphs with girth 5.
\newblock {\em Journal of Combinatorial Optimization}, 32:645--655, 2016.

\bibitem{dong-lin-2017}
Wei Dong and Wensong Lin.
\newblock On 2-distance coloring of plane graphs with girth 5.
\newblock {\em Discrete Applied Mathematics}, 217:495--505, 2017.

\bibitem{hartke}
Stephen~G Hartke, Sogol Jahanbekam, and Brent Thomas.
\newblock The chromatic number of the square of subcubic planar graphs.
\newblock {\em arXiv preprint arXiv:1604.06504}, 2016.

\bibitem{la-2021}
Hoang La.
\newblock 2-distance list ($\delta$+ 3)-coloring of sparse graphs.
\newblock {\em Graphs and Combinatorics}, 38(6):167, 2022.

\bibitem{la-mont-2022}
Hoang La and Mickael Montassier.
\newblock 2-distance list ($\delta$+ 2)-coloring of planar graphs with girth at
  least 10.
\newblock {\em Journal of Combinatorial Optimization}, 44(2):1356--1375, 2022.

\bibitem{molloy}
Michael Molloy and Mohammad~R. Salavatipour.
\newblock A bound on the chromatic number of the square of a planar graph.
\newblock {\em Journal of Combinatorial Theory, Series B}, 94:189--213, 2005.

\bibitem{thomassen}
Carsten Thomassen.
\newblock The square of a planar cubic graph is 7-colorable.
\newblock {\em Journal of Combinatorial Theory, Series B}, 128:192--218, 2018.

\bibitem{van-den}
Jan van~den Heuvel and Sean McGuinness.
\newblock Coloring the square of a planar graph.
\newblock {\em Journal of Graph Theory}, 42(2):110--124, 2003.

\bibitem{wegner}
Gerd Wegner.
\newblock Graphs with given diameter and a coloring problem.
\newblock {\em Technical report, University of Dormund}, 1977.

\bibitem{west}
Douglas~Brent West.
\newblock {\em Introduction to Graph Theory}, volume~2.
\newblock Prentice hall Upper Saddle River, 2001.

\end{thebibliography}

\end{document}